\documentclass[a4paper, 11pt]{article}
\usepackage{latexsym,amsfonts,palatino,eulervm,amsthm,graphicx,verbatim, float, mathdots, bm}
\usepackage[margin=2cm]{geometry}
\usepackage{amsmath}
\usepackage{algorithm}
\usepackage[noend]{algpseudocode}
\usepackage{amsthm}
\usepackage{amsfonts,mathrsfs}
\usepackage{multicol,soul}
\usepackage{multirow,bigdelim}
\usepackage{fancyvrb}
\usepackage{tkz-graph}
\usepackage[affil-it]{authblk}
\usepackage{hyperref}
\usepackage{xcolor}
\usepackage{nicematrix}
\usepackage{tikz}
\usepackage{mathtools}
\usepackage[inline]{enumitem}

\allowdisplaybreaks[2]

\newcommand{\cover}{\mathrel{\prec\!\!\cdot}}

\newcommand\SmallArray[1]{{%
  \tiny \arraycolsep=0.08\arraycolsep\ensuremath{\begin{array}{c}#1\end{array}}}}

\newcommand\Matrix[1]{{%
  \small \arraycolsep=0.8\arraycolsep\ensuremath{\begin{pmatrix}#1\end{pmatrix}}}}

\makeatletter
\def\BState{\State\hskip-\ALG@thistlm}
\makeatother

\algnewcommand\And{\textbf{and} }

\def\L{\mathcal{L}}
\def\C{\mathcal{C}}
\def\Ltwo{\mathcal{K}}

\newtheorem{theorem}{Theorem}[section]
\newtheorem{lemma}[theorem]{Lemma}
\newtheorem{corollary}[theorem]{Corollary}

\theoremstyle{remark}

\newtheorem*{thm}{Theorem}

\theoremstyle{definition}
\newtheorem{definition}[theorem]{Definition}
\newtheorem{example}[theorem]{Example}

\def\Ccal{\mathcal{C}}
\def\pos{\operatorname{pos}}
\def\sym{\operatorname{sym}}

\title{A Bruhat order for Latin squares and alternating sign hypermatrices}
\date{\vspace{-1.5cm}}
\author{%
  \href{mailto:angela.carnevale@universityofgalway.ie}{Angela Carnevale}%
  \footnote{School of Mathematical and Statistical Sciences, University of Galway, Ireland} and \href{mailto:obrien.cian@outlook.com}{Cian O'Brien}%
  \footnote{Department of Mathematics and Computer Studies, Mary Immaculate College, Ireland}}

\begin{document}
\setlength{\parindent}{0pt}
\setlength{\parskip}{1ex}

\maketitle

\begin{abstract}
\setlength{\parindent}{0pt}
\setlength{\parskip}{1ex}

\noindent The Bruhat order on permutation matrices extends to alternating sign
matrices via corner-sum matrices, where the order is given by
entrywise domination. A classical result of Lascoux and
Sch\"utzenberger states that alternating sign matrices form the
Dedekind--MacNeille completion of the Bruhat order on permutations.

Brualdi and Dahl introduced alternating sign hypermatrices as a
three-dimensional analogue of alternating sign matrices and used them
to generalise Latin squares, which may be viewed as three-dimensional
analogues of permutation matrices.

In this paper, in analogy with the two-dimensional case, we define and
study a Bruhat order $\preceq_B$ on Latin squares and alternating sign
hypermatrices.  We introduce the corresponding corner-sum
hypermatrices $\mathcal C_n$ and prove that entrywise domination on
$\mathcal C_n$ encodes this order. We show that $\mathcal C_n$ is a
distributive lattice, but that, unlike in dimension two, it is not the
Dedekind--MacNeille completion of the poset of Latin squares. We
further characterise the covering relations for $\C_n$ and prove rank formulae
generalising the classical case of alternating sign matrices. Finally,
we define monotone hypertriangles, prove that they are in bijection
with $\mathcal C_n$, and show that they also encode the order by
entrywise domination.

\end{abstract}

{\small

}

\section{Introduction}
This paper is devoted to the study of a new Bruhat order on Latin squares and alternating-sign hypermatrices, as natural three-dimensional analogues of permutations, alternating-sign matrices and their Bruhat order.
\subsection{The Bruhat order for permutations}\label{sec:bruhatperm}

A partially-ordered set, or \emph{poset}, is a set equipped with a partial order. A partial order $\preceq$ on a set $S$ is a relation that is reflexive ($x \preceq x$ for all $x \in S$), antisymmetric ($x \preceq y$ and $y \preceq x$ implies $x = y$), and transitive (if $x \preceq y$ and $y \preceq z$, then $x \preceq z$).  For $x,y \in S$, we write $x \prec y$ to denote $x\preceq y$ and $x\ne y$. An element $y \in S$ \emph{covers} $x \in S$, denoted $x\cover  y$, if $x \prec y$ and there is no $z \in S$ such that $x \prec z \prec y$. In a \emph{Hasse diagram}, each element is drawn above (i.e. with a higher vertical coordinate than) the elements they cover, and a line indicates a covering relation. A poset $(S, \preceq)$ is a \emph{lattice} if every pair of elements in $S$ has a unique least upper bound (supremum) and a unique greatest lower bound (infimum).

The set $S_n$ of permutations of $[n]:=\{1,2,\dots,n\}$ forms a poset under the \emph{Bruhat order} $\preceq_B$. We denote by $\sigma = \sigma_1\sigma_2\dots\sigma_n$ the permutation in $S_n$ satisfying $\sigma(i)=\sigma_i$ for $i\in[n]$. An \emph{inversion} in a permutation $\sigma$ is a pair $(\sigma_i,\sigma_j)$ for which $\sigma_i>\sigma_j$ and $i<j$. For example, the permutation $\sigma = 2314$ has two inversions; $(\sigma_1,\sigma_3) = (2,1)$ and $(\sigma_2,\sigma_3) = (3,1)$. Two permutations $\pi$ and $\sigma$ satisfy $\pi\preceq_B\sigma$ if $\pi$ can be obtained from $\sigma$ only by applying transpositions that decrease the number of inversions.

A permutation of $[n]$ can be naturally viewed as a permutation matrix. With such representation, an inversion in a permutation matrix $P$ is a (not necessarily contiguous) copy of $J_2 = \begin{pmatrix}
    0&1\\1&0
\end{pmatrix}$ as a submatrix of $P$. Hence, permutation matrices $P$ and $Q$ satisfy 
$P \preceq_B Q$ if and only if $P$ can be obtained from $Q$ by repeatedly replacing instances of $J_2$ with  $I_2 = \begin{pmatrix}
    1&0\\0&1
\end{pmatrix}$. For example,
\[P = \begin{pmatrix}0&0&1\\0&1&0\\1&0&0\end{pmatrix} \preceq_B \begin{pmatrix}0&1&0\\0&0&1\\1&0&0\end{pmatrix} = Q,\]
since $P$ results from removing the inversion in the upper-right corner of $Q$. The Hasse diagram of the permutation matrices of order 3 under the Bruhat order is displayed in Figure~\ref{fig:S3}.
\begin{figure}[h]
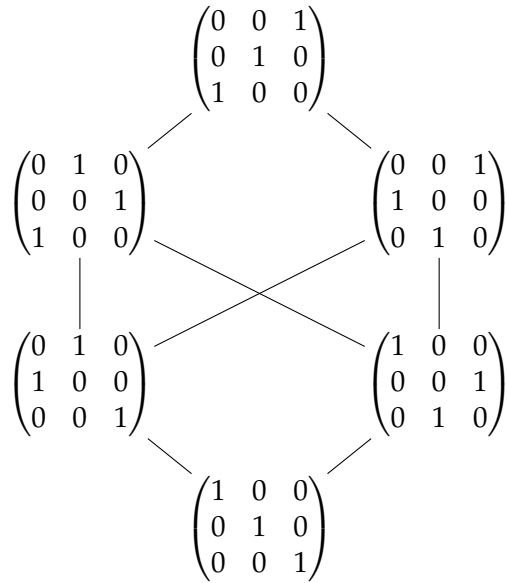

\begin{center}\begin{NiceTabular}{ccc}
&$\begin{pmatrix}0&0&1\\0&1&0\\1&0&0\end{pmatrix}$&\\

$\begin{pmatrix}0&1&0\\0&0&1\\1&0&0\end{pmatrix}$&\phantom{$\begin{pmatrix}0\\0\\0\\0\\0\end{pmatrix}$}&$\begin{pmatrix}0&0&1\\1&0&0\\0&1&0\end{pmatrix}$\\

$\begin{pmatrix}0&1&0\\1&0&0\\0&0&1\end{pmatrix}$&\phantom{$\begin{pmatrix}0\\0\\0\\0\\0\end{pmatrix}$}&$\begin{pmatrix}1&0&0\\0&0&1\\0&1&0\end{pmatrix}$\\

&$\begin{pmatrix}1&0&0\\0&1&0\\0&0&1\end{pmatrix}$&\\

\CodeAfter 
\tikz \draw  (1-2) -- (2-1); 
\tikz \draw  (1-2) -- (2-3); 
%%%
\tikz \draw  (3-1) -- (2-1); 
\tikz \draw  (3-1) -- (2-3); 
\tikz \draw  (3-3) -- (2-1); 
\tikz \draw  (3-3) -- (2-3); 
%%%
\tikz \draw  (3-1) -- (4-2); 
\tikz \draw  (3-3) -- (4-2); 
\end{NiceTabular}\end{center}
\caption{Hasse diagram of the Bruhat order on $3\times 3$ permutation matrices.}\label{fig:S3}
\end{figure}

For all $n\in\mathbb{N}$, the identity matrix $I_n$ is the unique minimal element of the poset of permutation matrices under the Bruhat order, and the complete inversion $J_n$ of $I_n$ is the unique maximal element. For $n\ge3$, however, $n\times n$ permutation matrices under the Bruhat order do not form a lattice.

\subsection{Alternating sign matrices and corner-sum matrices}\label{cornersum subsection}

The Bruhat order for permutation matrices can be equivalently defined as follows.
\begin{definition}\label{pm 2d bruhat def} Permutation matrices $P$ and $Q$ satisfy $P \preceq_B Q$ if and only if $P$ can be obtained from $Q$ by the repeated addition of  copies of
\[T=\begin{pmatrix}
    +1&-1\\-1&+1
\end{pmatrix}\]
to (not necessarily contiguous) $2\times2$ submatrices of $P$; see also~\cite{bruhat T}.\end{definition}
This alternative definition of the Bruhat order can be applied not only to permutation matrices, but more widely to \emph{alternating sign matrices}. An alternating sign matrix (ASM) is a $(0,1,-1)$-matrix for which the non-zero entries in each row and column alternate in sign, beginning and ending with $+1$. Where convenient, we represent each $1$ by $+$, each $-1$ by $-$, and omit the $0$ entries. The following is an example of a $4\times4$ ASM.
\[\begin{pmatrix}
    &+&&\\
    &&+&\\
    +&-&&+\\
    &+&&
\end{pmatrix}=\begin{pmatrix}
    0&1&0&0\\
    0&0&1&0\\
    1&-1&0&1\\
    0&1&0&0
\end{pmatrix}\]

An immediate consequence of the definition is that ASMs are square matrices with all row and column sums equal to 1. All permutation matrices are ASMs, and ASMs form the \emph{Dedekind-MacNeille completion} of the permutation matrices under the Bruhat order \cite[Theorem~4.4]{lattice}, which is the smallest lattice into which permutations under the Bruhat order embed. In the lattice of ASMs, $A$ covers $B$ if and only if $B$ can be obtained from $A$ by the addition of a single contiguous copy of $T$.

Set $[0,n]:=\{0\}\cup [n]$. A \emph{corner-sum matrix} of order $n$ is an $(n+1)\times(n+1)$ matrix such that in row and column $k \in[0,n]$, the first entry is 0, the last entry is $k$, and each  entry is either equal to or 1 greater than the previous going from left to right and from top to bottom. The set of corner-sum matrices is in bijection with ASMs \cite{corner-sum}. Figure~\ref{hasse3} shows the bijection for all $3\times3$ ASMs and corner-sum matrices of order $3$.

 Given an $m \times n$ matrix $A$, define the \emph{corner-sum} $\Sigma(A)$ of $A$ to be the $(m+1) \times (n+1)$ matrix with the following $(i,j)$-entry, for all $i \in [0,m]$ and $j \in [0,n]$.
\[\Sigma(A)_{i,j} = \sum_{a = 1}^{i} \sum_{b=1}^j A_{a,b}.\]

A corner-sum matrix $C$ is \emph{the} corner-sum matrix of an ASM $A$ of order $n$ if $C = \Sigma(A)$. 
To see this correspondence is a bijection, note that the entries of $A$ can be recovered from $C = \Sigma(A)$ by $$A_{i,j} = C_{i,j} - C_{i-1,j} - C_{i,j-1} + C_{i-1,j-1}$$ for all $1 \leq i, j \leq n$, and consider the square submatrix of entries from $(i-1,j-1)$ to $(i,j)$.
\begin{itemize}
\item If all four entries are equal, then $A_{i,j} = C_{i,j} - C_{i-1,j} - C_{i,j-1} + C_{i-1,j-1} = 0$.
\item If all are equal except $C_{i,j}$, then $A_{i,j} = C_{i,j} - C_{i-1,j} - C_{i,j-1} + C_{i-1,j-1} = 1$.
\item If $C_{i,j} = C_{i,j-1} > C_{i-1,j-1} = C_{i-1,j}$, then $A_{i,j} = C_{i,j} - C_{i-1,j} - C_{i,j-1} + C_{i-1,j-1} = 0$.
\item If $C_{i,j} = C_{i,j-1} = C_{i-1,j} > C_{i-1,j-1}$, then $A_{i,j} = C_{i,j} - C_{i-1,j} - C_{i,j-1} + C_{i-1,j-1} = -1$.
\item If $C_{i,j} > C_{i,j-1} = C_{i-1,j} > C_{i-1,j-1}$, then $A_{i,j} = C_{i,j} - C_{i-1,j} - C_{i,j-1} + C_{i-1,j-1} = 0$.
\end{itemize}

The Bruhat order on ASMs can be characterised using their corner-sum matrices  \cite[Theorem 1]{bruhat cs}. Indeed, if $A$ and $B$ are ASMs, then
\begin{equation}\label{eq:cornersumASM}
A \preceq_B B \iff \Sigma(A) \geq \Sigma(B),\end{equation}
where $\geq$ indicates that one matrix is greater than or equal to another entrywise. In the corresponding lattice, $C$ covers $D$ if and only if $C-D$ has a single non-zero entry, and its value is equal to $1$.

Moreover, the supremum and infimum of any pair of alternating sign matrices can easily be determined using their corner-sum matrices.

\begin{example}\label{sup inf example}
\[A = \begin{pmatrix}0&1&0\\1&0&0\\0&0&1\end{pmatrix} \implies \Sigma(A) = \begin{pmatrix}0&0&0&0\\0&0&1&1\\0&1&2&2\\0&1&2&3\end{pmatrix} \text{ and } B = \begin{pmatrix}1&0&0\\0&0&1\\0&1&0\end{pmatrix} \implies \Sigma(B) = \begin{pmatrix}0&0&0&0\\0&1&1&1\\0&1&1&2\\0&1&2&3\end{pmatrix}\]

Clearly the unique matrix $X$ with least entries such that $X \geq \Sigma(A)$ and $X \geq \Sigma(B)$ is the matrix $X = \max(\Sigma(A), \Sigma(B))$ with the entrywise maximum values of $\Sigma(A)$ and $\Sigma(B)$. Similarly, the unique matrix $Y$ with greatest entries such that $Y \leq \Sigma(A)$ and $Y \leq \Sigma(B)$ is the matrix $Y = \min(\Sigma(A), \Sigma(B))$ with the entrywise minimum values of $\Sigma(A)$ and $\Sigma(B)$. 
\[\max(\Sigma(A), \Sigma(B)) = \begin{pmatrix}0&0&0&0\\0&1&1&1\\0&1&2&2\\0&1&2&3\end{pmatrix}  \text{ and } \min(\Sigma(A), \Sigma(B)) =\begin{pmatrix}0&0&0&0\\0&0&1&1\\0&1&1&2\\0&1&2&3\end{pmatrix}\]

Inverting the corner-sum operation on the min and max matrices then gives the supremum and infimum, respectively, of $A$ and $B$ under the Bruhat order.
\[\Sigma^{-1}\big(\max(\Sigma(A), \Sigma(B))\big) = \begin{pmatrix}1&0&0\\0&1&0\\0&0&1\end{pmatrix} \text{ and } \Sigma^{-1}\big(\min(\Sigma(A), \Sigma(B))\big) =\left(\begin{array}{rrr}0&1&0\\1&-1&1\\0&1&0\end{array}\right)\]
\end{example}

\begin{figure}[h]
\includegraphics[width=\textwidth]{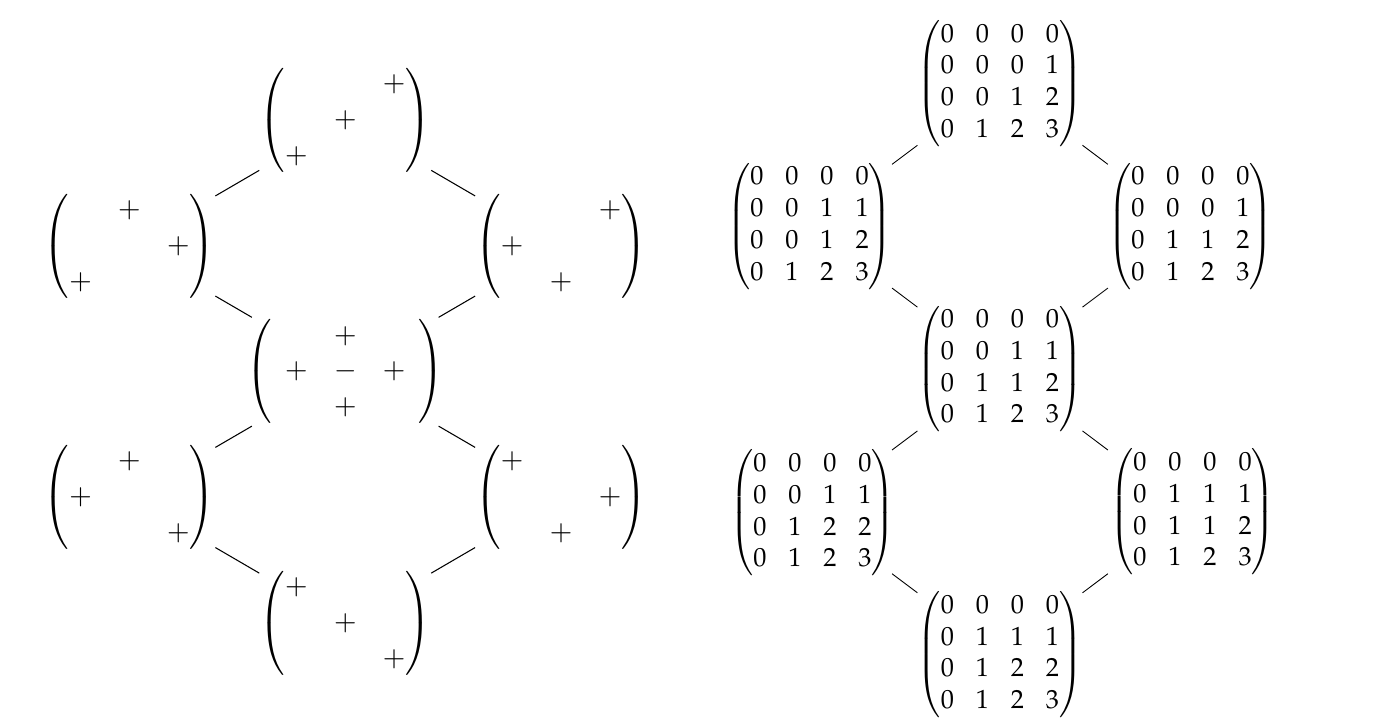}\caption{The Hasse diagrams for the ASMs (left) and corner-sum matrices (right) of order 3 under the Bruhat order.}
\label{hasse3}
\end{figure}

The corner-sum characterisation of the Bruhat order can also be described for permutations, providing useful criteria for the Bruhat relation \cite[Theorem~2.1.5]{comb. coxeter}. We adapt   \cite[Eq.~(2.3)]{comb. coxeter} to match the definition of a corner-sum matrix as follows. Given $\sigma\in S_n$, define $\sigma[i,j] = |\{a\in[j]:\:\sigma(j)\le i\}|$ for all $i,j\in[n]$.

For example, the following table shows all values of $\sigma[i,j]$ for $\sigma = 312$. Shown also are the permutation matrix $P$ with $(\sigma(i),i)$-entry equal to 1, for $i\in[3]$, and $\Sigma(P)$.
\[\begin{array}{l|ccc|}
    & 3\le i & 3,1 \le i & 3,1,2\le i\\
    \hline
    i=1   	& 0 & 1 & 1 \\
    i=2     & 0 & 1 & 2 \\
    i=3     & 1 & 2 & 3 \\
    \hline
\end{array}\hspace{0.3cm}
P = \begin{pmatrix}
    &+&\\
    &&+\\
    +&&
\end{pmatrix}\hspace{0.3cm}
\Sigma(P) = \begin{pmatrix}
    0&0&0&0\\
    0&0&1&1\\
    0&0&1&2\\
    0&1&2&3
\end{pmatrix}\]

For $\sigma,\pi\in S_n$, it follows that $\sigma \preceq_B \pi$ if and only if  $\sigma[i,j]\ge\pi[i,j]$ for all $i,j\in[n]$.

\subsection{Latin squares and hypermatrices}

A \emph{Latin square} of order $n$ is an $n \times n$ array containing $n$ symbols such that each symbol occurs exactly once in each row and column. Throughout the paper, we use $[n]$ as symbols. We denote by $\L_n$ the set of all Latin squares of order $n$. Throughout the paper, we use $\Sigma(L)$ for the corner-sum matrix of $L\in\L_n$.

Fernandes, da Cruz, and Salomão \cite{Latin bruhat} defined a Bruhat order on Latin squares  $L_1, L_2 \in \L_n$ such that $L_1$ precedes $L_2$ if and only if $\Sigma(L_1) \ge \Sigma(L_2)$. In this paper, we propose an alternative definition for the Bruhat order on Latin squares. We do this by interpreting Latin squares as the 3-dimensional analogues of permutation matrices, and generalise the definition in a way that is consistent with this interpretation. While both definitions are related, our definition more closely aligns with the inversion definition of the Bruhat order  on permutations. Throughout the paper, where appropriate, we highlight relationships and differences between these two definitions.

A hypermatrix $A = [A_{ijk}]$ of order $n$ is an $n \times n \times n$ array with $n^2$ lines of each of the three following types. Each line has $n$ entries.
\begin{itemize}
\item \emph{Rows} $A_{i*k} = [A_{ijk} : j = 1, \dots, n]$, for given $1 \leq i,k \leq n$;
\item \emph{Columns} $A_{*jk} = [A_{ijk} : i = 1, \dots, n]$, for given $1 \leq j,k \leq n$;
\item \emph{Vertical lines} $A_{ij*} = [A_{ijk} : k = 1, \dots, n]$, for given $1 \leq i,j \leq n$.
\end{itemize}

We define the $k$th \emph{plane} of $A$ to be the horizontal plane $A_{**k} = [A_{ijk}: 1 \leq i,j \leq n ]$, for given $k \in [n]$. Permutation hypermatrices are $(0,1)$-hypermatrices with exactly one non-zero entry in each row, column, and vertical line.

A Latin square $L\in\L_n$ can be thought of as representing a permutation hypermatrix $P$, where the positions of the symbol $k$ in $L$ correspond to the non-zero entries in the $k$th plane $P_k$ of $P$. In a permutation hypermatrix, each plane is a permutation matrix, and no pair of planes have non-zero entries in common. The following weighted sum provides a bijection between Latin squares and permutation hypermatrices of order $n$.
\[L = L(P) = \sum_{k=1}^n kP_k\]

\begin{example} The following Latin square corresponds to the weighted sum of planes of a unique permutation hypermatrix.
\[\begin{array}{|c|c|c|}
\hline
1 & 2 & 3 \\
\hline
2 & 3 & 1 \\
\hline
3 & 1 & 2\\
\hline
\end{array}
\longleftrightarrow
1\Matrix{
      1 & 0 & 0 \\
      0 & 0 & 1 \\
      0 & 1 & 0}
+2\Matrix{
      0 & 1 & 0 \\
      1 & 0 & 0 \\
      0 & 0 & 1}
+3\Matrix{
      0 & 0 & 1 \\
      0 & 1 & 0 \\
      1 & 0 & 0}\]\end{example}

This interpretation as permutation hypermatrices leads very naturally to two generalisations of Latin squares, first introduced by Brualdi and Dahl \cite{brualdidahl}.

\begin{itemize}
\item An \emph{alternating sign hypermatrix (ASHM)} of order $n$ is an $n\times n \times n$ hypermatrix with entries in $\{0,1,-1\}$ such that the non-zero entries in each row, column, and vertical line alternate in sign, starting and ending with $+1$.
\item A \emph{planar alternating sign hypermatrix (PASHM)} of order $n$ is an $n\times n \times n$ hypermatrix with entries in $\{0,1,-1\}$ such that each plane $P_k$ of $A$ is an ASM and $P_1 + P_2 + \dots + P_n$ is the all-ones matrix.
\end{itemize}
Note that all ASHMs are PASHMs, but not all PASHMs are ASHMs. We use $P_k\nearrow P_{k+1}$ to denote that the plane $P_k$ occurs directly before $P_{k+1}$.

\begin{example}\label{3x3x3} 
The following is a PASHM of order $3$, but not an ASHM
\[A = \Matrix{
      0 & 1 & 0 \\
      1 & -1 & 1 \\
      0 & 1 & 0}
\hspace{-0.1cm}\nearrow\hspace{-0.1cm}\Matrix{
      0 & 0 & 1 \\
      0 & 1 & 0 \\
      1 & 0 & 0}
\hspace{-0.1cm}\nearrow\hspace{-0.1cm}\Matrix{
      1 & 0 & 0 \\
      0 & 1 & 0 \\
      0 & 0 & 1}.\]
The following is an ASHM of order $3$, and therefore also a PASHM
\[B = \Matrix{
      0 & 0 & 1 \\
      0 & 1 & 0 \\
      1 & 0 & 0}
\hspace{-0.1cm}\nearrow\hspace{-0.1cm}\Matrix{
      0 & 1 & 0 \\
      1 & -1 & 1 \\
      0 & 1 & 0}
\hspace{-0.1cm}\nearrow\hspace{-0.1cm}\Matrix{
      1 & 0 & 0 \\
      0 & 1 & 0 \\
      0 & 0 & 1}.\]
\end{example}

For compactness, and for easier comparison to Latin squares, we use an alternative notation introduced in \cite{ASHM polytope}, representing a PASHM $A$ of order $n$ by an $n \times n$ grid in which cell $(i,j)$ of the grid contains the formal sum of the non-zero entries of the vertical line $A_{ij*}$ of the PASHM times their index in the vertical line (i.e.\ representing the summands without adding them together). Using this notation, we represent the PASHMs from Example~\ref{3x3x3} in the following way

\begin{equation}\label{eq:AB}
A = \begin{array}{|c|c|c|}
\hline
    3 & 1 & 2 \\
\hline
    1 & \SmallArray{-1+2\\+3} & 1 \\
\hline
    2 & 1 & 3 \\
\hline
\end{array}
\hspace{1cm}\text{and}\hspace{1cm}
B = \begin{array}{|c|c|c|}
\hline
    3 & 2 & 1 \\
\hline
    2 & \SmallArray{1-2\\+3} & 2 \\
\hline
    1 & 2 & 3 \\
\hline
\end{array}\,.
\end{equation}

It can be easily checked if an array of this form is (a) a PASHM or (b) an ASHM as follows.
\begin{itemize}
\item [(a)] For each $k\in[n]$, the positions of the signs corresponding to $k$ in a PASHM of order $n$ form an ASM. Further, in each cell of a PASHM, there is one more positive entry than negative.
\item[(b)] In addition to the conditions in (a), in an ASHM, the signs of the terms in each cell alternate, beginning and ending with $+$, when listed in increasing order.
\end{itemize}

Note that while any permutation of the planes of an ASHM results in a PASHM, it is not true that every PASHM is some permutation of the planes of an ASHM.

\begin{example} The following PASHMs have no permutation to an ASHM.
\[\begin{array}{|c|c|c|c|}
\hline
1	&2				&3				&4	\\
\hline
3	&1				&\SmallArray{2-3\\+4}	&3	\\
\hline
2	&\SmallArray{-2+3\\+4}	&1				&2	\\
\hline
4	&2				&3				&1	\\
\hline
\end{array}\hspace{2cm}
\begin{array}{|c|c|c|c|}
\hline
1	&2				&3				&4	\\
\hline
2	&\SmallArray{-2+3\\+4}	&\SmallArray{1+2\\-3}	&3	\\
\hline
4	&\SmallArray{1+2\\-4}	&\SmallArray{-2+3\\+4}	&2	\\
\hline
3	&4				&2				&1	\\
\hline
\end{array}\hspace{2cm}
\begin{array}{|c|c|c|c|}
\hline
1	&2				&3				&4	\\
\hline
2	&\SmallArray{-2+3\\+4}	&\SmallArray{1+2\\-3}	&3	\\
\hline
4	&\SmallArray{1+2\\-4}	&\SmallArray{-1+3\\+4}	&1	\\
\hline
3	&4				&1				&2	\\
\hline
\end{array}\]
\end{example}

\medskip

The rest of this paper is devoted to the definition and the study of a new Bruhat order on Latin squares, ASHMs, and related objects. In the next subsection, we give an outline of the paper together with a description of its main results.
\subsection{Main results and outline of the paper} In
Section~\ref{sec:BLS}, we use the interpretation of Latin squares as
permutation hypermatrices to define a Bruhat order $\preceq_B$ on the
set $\L_n$ of Latin squares of order $n$ and on ASHMs of order~$n$. Our definition uses a suitable notion of $T$-blocks, which allows
us to compare both Latin squares and ASHMs. We also compare our Bruhat
order with the one defined in \cite{Latin bruhat}, and analyse the
effects of classical Latin square moves (so-called cycle switches) in
relation to our order.

In order to extend the (Latin square and) ASHM Bruhat poset to a
lattice, in Section~\ref{sec:CSH} we introduce a 3-dimensional
analogue of corner-sum matrices, namely \emph{corner-sum hypermatrices}
and extend our Bruhat order to these.  This perspective allows us to
prove another characterisation of the Bruhat relation for Latin
squares (Theorem~\ref{decreasing theorem}) and in particular of the
covering relation (Theorem~\ref{thm:coverLS}). The rest of the section
is devoted to the study of the set~$\C_n$ of corner-sum hypermatrices
of order $n$ under the Bruhat order, which we also refer to as the corner-sum Latin lattice in light of the following main result; cf.~Theorems~\ref{thm:CSlattice} and Corollary~\ref{cor:FDL}.
\begin{thm} Corner-sum hypermatrices of order $n$ form a 
distributive lattice under the Bruhat order.  \end{thm}

Section~\ref{sec:invrank} is devoted to a further study of the
corner-sum Latin lattice: we characterise maximal and minimal elements, and its rank function in analogy with the 2-dimensional case (see Theorems~\ref{thm:rankL} and \ref{thm:rankCn}).

In Section~\ref{sec:DM} we prove that the corner-sum Latin lattice is not the Dedekind-MacNeille completion of the Latin square (and ASHM) Bruhat poset, breaking the analogy with the 2-dimensional case. We show that the same is true when our operation for the construction of the Latin lattice is applied to the Bruhat order from \cite{Latin bruhat}.

Section~\ref{sec:MH} is devoted to the definition of another related family of objects: monotone hypertriangles. We give a definition and some properties in analogy with the 2-dimensional case, and we show \begin{enumerate*}\item[(1)] that monotone hypertriangles are in bijection with elements of the lattice $\mathcal C_n$ and \item[(2)] that the Bruhat order on the lattice is encoded by entrywise comparison on these objects as well.\end{enumerate*}

We conclude in Section~\ref{sec:conclusions} with some remarks on the enumeration of the objects studied throughout the paper and some open problems.

%%%%%%%%%%%%%%%%%%%%%%%%%%%%%%%%%%%%%%%%%%

\section{A Bruhat order for Latin squares and ASHMs}\label{sec:BLS}
Recall the definition of Bruhat order on permutations from Section~\ref{sec:bruhatperm}. Note that the symmetry of $I_2$ and $J_2$ implies that interchanging these matrices can be equivalently seen as swapping their rows or swapping their columns. Partly inspired by this observation, we will define a new Bruhat order for Latin squares (see Definition~\ref{BLS}). Before giving the full definition of the new Bruhat order, we describe some features it should have to be a suitable 3-dimensional generalisation of the one on permutations. In the new order, for instance, replacing instances of $J_2 \nearrow I_2$ with $I_2 \nearrow J_2$ as follows
\[\begin{pmatrix}&+\\+&\end{pmatrix}\nearrow \begin{pmatrix}+&\\&+\end{pmatrix}\longrightarrow \begin{pmatrix}+&\\&+\end{pmatrix}\nearrow \begin{pmatrix}&+\\+&\end{pmatrix}\]
essentially plays the same role as replacing copies of $J_2$ with copies of $I_2$ in the $2$-dimensional case. Note that the former can be interpreted as  swapping rows, columns, or planes.

For $L_1,L_2\in\L_n$, this means that $L_1$ precedes $L_2$ in our new Bruhat order if, for instance, $L_1$ can be obtained from $L_2$ by making successive subarray replacements of the form
\[\begin{array}{|c|c|}\hline b	&	a\\ \hline a	&	b\\ \hline\end{array} \longrightarrow \begin{array}{|c|c|}\hline a	&	b\\ \hline b	&	a\\ \hline\end{array}\;,\]
where $a < b$ and the  subarray is not necessarily contiguous.  In the language of Latin squares, such a replacement (without the requirement that $a < b$) is known as an \emph{intercalate} switch \cite{wanless04}. With the requirement $a<b$, we call this a \emph{decreasing} intercalate switch.

\begin{example} 
Let
\[A = \begin{array}{|c|c|c|c|}
\hline
1	&	2	&	3	&	4\\
\hline
2	&	1	&	4	&	3\\
\hline
3	&	4	&	1	&	2\\
\hline
4	&	3	&	2	&	1\\
\hline
\end{array}\,,\hspace{0.5cm}
B = \begin{array}{|c|c|c|c|}
\hline
2	&	1	&	3	&	4\\
\hline
1	&	2	&	4	&	3\\
\hline
3	&	4	&	1	&	2\\
\hline
4	&	3	&	2	&	1\\
\hline
\end{array}\,,\hspace{0.5cm}
C = \begin{array}{|c|c|c|c|}
\hline
1	&	2	&	3	&	4\\
\hline
2	&	4	&	1	&	3\\
\hline
3	&	1	&	4	&	2\\
\hline
4	&	3	&	2	&	1\\
\hline
\end{array}\,,\hspace{0.5cm}
D = \begin{array}{|c|c|c|c|}
\hline
1	&	4	&	3	&	2\\
\hline
2	&	1	&	4	&	3\\
\hline
3	&	2	&	1	&	4\\
\hline
4	&	3	&	2	&	1\\
\hline
\end{array}\,.\]
Then $A$ precedes $B$, $C$, and $D$ in the new Bruhat order. All other pairs are incomparable.
\end{example}

We now give the full definition of our new Bruhat order for Latin squares.
\begin{definition}\label{BLS}
Let $L_1,L_2\in \L_n$. We set $L_1\preceq_B L_2$ if and only if $L_1$ can be obtained from $L_2$ by the repeated addition of \emph{T-blocks} of the following form
\[T = \begin{array}{|c|c|}
\hline
a-b & b-a\\
\hline
b-a & a-b\\
\hline
\end{array}\;,\]
where $a < b$ and the grid is not necessarily contiguous. The Bruhat order on Latin squares of order $n$ is the poset $(\L_n,\preceq_B).$
\end{definition}

This definition includes cases where copies of $J_2 \nearrow I_2$ are replaced with $I_2 \nearrow J_2$, as described before, while also allowing Latin squares not containing $J_2 \nearrow I_2$ or $I_2 \nearrow J_2$ to be compared.

\begin{example}The following Latin squares satisfy $L_1 \preceq_B L_2$.

\[L_2 = \begin{array}{|c|c|c|}
\hline
3	&	1	&	2\\
\hline
1	&	2	&	3\\
\hline
2	&	3	&	1\\
\hline
\end{array}\,, \hspace{1cm} L_2 + \begin{array}{|c|c|c|}
\hline
1-3	&	3-1	&	0\\
\hline
3-1	&	1-3	&	0\\
\hline
0	&	0	&	0\\
\hline
\end{array}+\begin{array}{|c|c|c|}
\hline
0	&	0	&	0\\
\hline
2-3	&	3-2	&	0\\
\hline
3-2	&	2-3	&	0\\
\hline
\end{array}=\begin{array}{|c|c|c|}
\hline
1	&	3	&	2\\
\hline
2	&	1	&	3\\
\hline
3	&	2	&	1\\
\hline
\end{array} = L_1\]
\end{example}

In analogy with the $2$-dimensional case,  permutation hypermatrices $P$ and $Q$ satisfy $P \preceq_B Q$ if and only if $P$ can be obtained by the addition of 3-dimensional positive T-blocks \cite{ASHM-decomp} of the form %$T = I_2 \nearrow J_2$ 
\[T=\begin{pmatrix}
    +&-\\-&+
\end{pmatrix}\nearrow\begin{pmatrix}
    -&+\\+&-
\end{pmatrix}\]
to $2\times2\times2$ sub-hypermatrices of $Q$. Note a negative T-block has the form $-T$.

In a Latin square, swapping a subset of two rows (resp.\ columns or symbols) occupying an equal number of columns and symbols (resp.\ rows and symbols, or rows and columns) results in a Latin square. A minimal such swap is known as a \emph{cycle switch} \cite{wanless04}. Intercalates are cycle switches of order 2, and the addition of a contiguous T-block corresponds to an intercalate switch of consecutive integers.

\begin{lemma}\label{contiguous T-blocks}
Let $T$ be a positive T-block. Then $T$ is the sum of contiguous positive T-blocks.
\end{lemma}
\begin{proof}
Let $n_i,n_j, n_k\ge1$ be the differences between the indices of the nonzero entries in $T$ in the $i$, $j$, and $k$ directions respectively. We prove the result by induction on $n_i + n_j + n_k$. Clearly, the claim is true for $n_i = n_j = n_k = 1$. Now let $n\ge4$ and assume the claim is true for all $T$-blocks with $n_i + n_j + n_k < n$. Suppose $n_i + n_j + n_k = n$. Without loss of generality, suppose $n_k > 1$, and consider the T-block $T'$ of dimension $n_i \times n_j \times (n_k-1)$. By induction, $T'$ is the sum of positive contiguous T-blocks. Let $S$ be the sum of $T'$ and each positive contiguous T-block occupying planes $n_k-1$ and $n_k$ of $T$. Each entry of $S$ which is not in one of the corners of planes $n_k-1$ or $n_k$ is  equal to 0, since it results from the sum of either 2 or 4 entries of positive contiguous T-blocks (half of which are $+1$ and half $-1$). The entries in the corners of plane $n_k-1$ of $T'$ are either $+1$ or $-1$, and exactly one entry of a positive contiguous T-block is added to give 0 in the corresponding position of $S$. Similarly, exactly one entry of a positive contiguous T-block is added to each corner of plane $n_k$. Therefore $S = T$, and $T$ is the sum of contiguous positive T-blocks.
\end{proof}
\begin{corollary}\label{contiguous T corollary}
    Let $L_1,L_2\in\L_n$. Then $L_1 \preceq_BL_2$ if and only if $L_1$ can be obtained from $L_2$ by the repeated addition of \emph{contiguous} positive T-blocks. Moreover, if $L_1$ can be obtained from $L_2$ by a single contiguous decreasing intercalate switch of consecutive integers% (the addition of a single contiguous positive T-block)
    , then $L_1 \cover_B L_2$.
\end{corollary}

Given $n \times n$ Latin squares $A$ and $B$, we let $A \preceq' B$ if $B$ precedes $A$ in the Bruhat order of \cite{Latin bruhat}. Note that $A \preceq_BB$ if and only if $A\preceq'B$ for $n \le 3$ (meaning our definition of the Bruhat order and that of \cite{Latin bruhat} are the duals of one another). For $n \ge 4$, the definitions differ more substantially.

\begin{example}\label{Comparing Bruhat definitions} Pages 4--5 of \cite{Latin bruhat} define the following Latin squares $A,C,D$ and note that $A \preceq'D\preceq' C$. Our definition similarly has $A\preceq_BC$ and $D\preceq_BC$, but $A$ and $D$ are incomparable.
\[A = \begin{array}{|c|c|c|c|}
\hline
1	&	2	&	3	&	4\\
\hline
4	&	3	&	1	&	2\\
\hline
3	&	4	&	2	&	1\\
\hline
2	&	1	&	4	&	3\\
\hline
\end{array},\hspace{0.5cm}
C = \begin{array}{|c|c|c|c|}
\hline
2	&	1	&	3	&	4\\
\hline
4	&	3	&	1	&	2\\
\hline
3	&	4	&	2	&	1\\
\hline
1	&	2	&	4	&	3\\
\hline
\end{array},\hspace{0.5cm}
D = \begin{array}{|c|c|c|c|}
\hline
2	&	1	&	3	&	4\\
\hline
3	&	4	&	1	&	2\\
\hline
4	&	3	&	2	&	1\\
\hline
1	&	2	&	4	&	3\\
\hline
\end{array}\]
\end{example}

The Bruhat order $\preceq_B$ on $3 \times 3$ Latin squares is represented by the Hasse diagram in Figure~\ref{fig:L3}.

\begin{figure}[h]
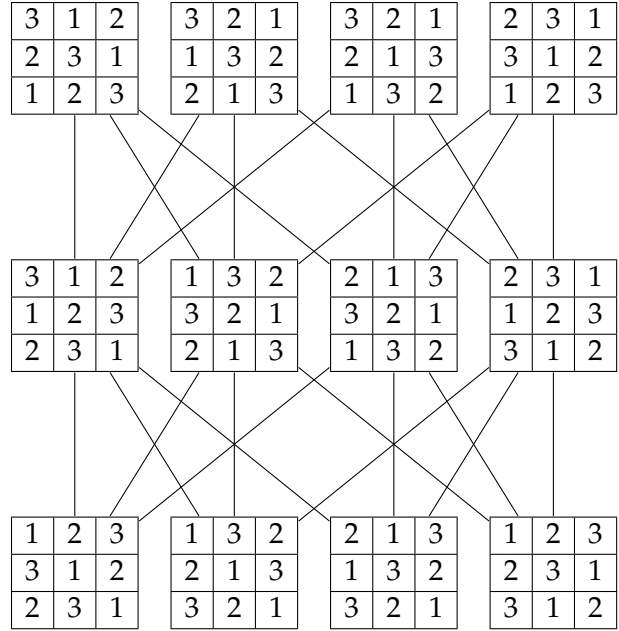

\begin{center}\quad\qquad\qquad\begin{NiceTabular}{cccc}

$\begin{array}{|c|c|c|}
\hline
3	&	1	&	2\\
\hline
2	&	3	&	1\\
\hline
1	&	2	&	3\\
\hline
\end{array}$
&
$\begin{array}{|c|c|c|}
\hline
3	&	2	&	1\\
\hline
1	&	3	&	2\\
\hline
2	&	1	&	3\\
\hline
\end{array}$
&
$\begin{array}{|c|c|c|}
\hline
3	&	2	&	1\\
\hline
2	&	1	&	3\\
\hline
1	&	3	&	2\\
\hline
\end{array}$
&
$\begin{array}{|c|c|c|}
\hline
2	&	3	&	1\\
\hline
3	&	1	&	2\\
\hline
1	&	2	&	3\\
\hline
\end{array}$

\\

&\phantom{$\begin{pmatrix}0\\0\\0\\0\end{pmatrix}$}&&

\\

$\begin{array}{|c|c|c|}
\hline
3	&	1	&	2\\
\hline
1	&	2	&	3\\
\hline
2	&	3	&	1\\
\hline
\end{array}$
&
$\begin{array}{|c|c|c|}
\hline
1	&	3	&	2\\
\hline
3	&	2	&	1\\
\hline
2	&	1	&	3\\
\hline
\end{array}$
&
$\begin{array}{|c|c|c|}
\hline
2	&	1	&	3\\
\hline
3	&	2	&	1\\
\hline
1	&	3	&	2\\
\hline
\end{array}$
&
$\begin{array}{|c|c|c|}
\hline
2	&	3	&	1\\
\hline
1	&	2	&	3\\
\hline
3	&	1	&	2\\
\hline
\end{array}$

\\

&&\phantom{$\begin{pmatrix}0\\0\\0\\0\end{pmatrix}$}&

\\

$\begin{array}{|c|c|c|}
\hline
1	&	2	&	3\\
\hline
3	&	1	&	2\\
\hline
2	&	3	&	1\\
\hline
\end{array}$
&
$\begin{array}{|c|c|c|}
\hline
1	&	3	&	2\\
\hline
2	&	1	&	3\\
\hline
3	&	2	&	1\\
\hline
\end{array}$
&
$\begin{array}{|c|c|c|}
\hline
2	&	1	&	3\\
\hline
1	&	3	&	2\\
\hline
3	&	2	&	1\\
\hline
\end{array}$
&
$\begin{array}{|c|c|c|}
\hline
1	&	2	&	3\\
\hline
2	&	3	&	1\\
\hline
3	&	1	&	2\\
\hline
\end{array}$
\\

\CodeAfter 
\tikz \draw  (1-1) -- (3-1); 
\tikz \draw  (1-1) -- (3-2); 
\tikz \draw  (1-2) -- (3-1); 
\tikz \draw  (1-2) -- (3-2); 
%%%
\tikz \draw  (1-1) -- (3-3); 
\tikz \draw  (1-2) -- (3-4);
\tikz \draw  (1-3) -- (3-1); 
\tikz \draw  (1-4) -- (3-2);  
%%%
\tikz \draw  (1-3) -- (3-3); 
\tikz \draw  (1-3) -- (3-4); 
\tikz \draw  (1-4) -- (3-3); 
\tikz \draw  (1-4) -- (3-4); 
%%%
\tikz \draw  (5-1) -- (3-1); 
\tikz \draw  (5-1) -- (3-2); 
\tikz \draw  (5-2) -- (3-1); 
\tikz \draw  (5-2) -- (3-2); 
%%%
\tikz \draw  (5-1) -- (3-3); 
\tikz \draw  (5-2) -- (3-4);
\tikz \draw  (5-3) -- (3-1); 
\tikz \draw  (5-4) -- (3-2);  
%%%
\tikz \draw  (5-3) -- (3-3); 
\tikz \draw  (5-3) -- (3-4); 
\tikz \draw  (5-4) -- (3-3); 
\tikz \draw  (5-4) -- (3-4); 
\end{NiceTabular}\end{center}
\caption{The Hasse diagram of $\L_3$.}\label{fig:L3}
\end{figure}

In \cite{wanless04}, it is shown that for arbitrarily large $n \in \mathbb{N}$, there exist pairs of Latin squares of order $n$ which cannot be obtained from one another by a sequence of cycle switches. Cycle switches, however, are useful to describe the Bruhat relation for Latin squares, and to relate it to the 2-dimensional Bruhat order.

\begin{theorem}\label{row/col/sym swap} Let $L_1,L_2\in\L_n$ be such that $L_1$ can be obtained from $L_2$ by repeated cycle switches. If the (partial) permutation with lower index in each switch precedes the (partial) permutation it is switched with in the 2-dimensional Bruhat order, then $L_1\preceq_BL_2$.
\end{theorem}

\begin{proof}
    Without loss of generality, suppose $L_1$ and $L_2$ differ by a switch of a subset of 2 rows, $R_1$ and $R_2$. To obtain the entries of $L_1$ from the entries of $L_2$ in $R_1$, there are integers $a<b$ for which $a-b$ is being added to an entry of $R_1$ and $b-a$ is being added, in the same row, to an entry with higher index. Respectively, $b-a$ and $a-b$ are being added to the corresponding entries of $R_2$. By Definition~\ref{pm 2d bruhat def}, this means the permutation represented by $R_1$ precedes the permutation represented by $R_2$ in the 2-dimensional Bruhat order. If $L_1$ and $L_2$ differ by multiple such switches, each intermediate Latin square precedes the previous and is preceded by the next in the new Bruhat order.
\end{proof}

Theorem \ref{row/col/sym swap} demonstrates that the relationship between our 3-dimensional Bruhat order on $n \times n$ Latin squares and the 2-dimensional Bruhat order on $S_n$ is strongly analogous to the relationship between the 2-dimensional Bruhat order on $S_n$ and the natural order on $[n]$, which could be seen as its 1-dimensional analogue. If there is a sequence of transpositions from $\sigma$ to $\pi$,
each swapping entries in positions $i<j$ with the smaller value
initially in position $i$, then $\sigma \preceq_B \pi$.

As in the 2-dimensional case, %the T-block definition allows us to 
 we may compare ASHMs with Latin squares under the Bruhat order. Unlike the 2-dimensional case, however, ASHMs do not extend the Latin poset to a lattice. In the $3\times3$ case, for example, extending from Latin squares to ASHMs results in a poset containing a unique minimal and unique maximal element which is not a lattice, see Figure~\ref{ASHM3}.

\begin{figure}[htb]
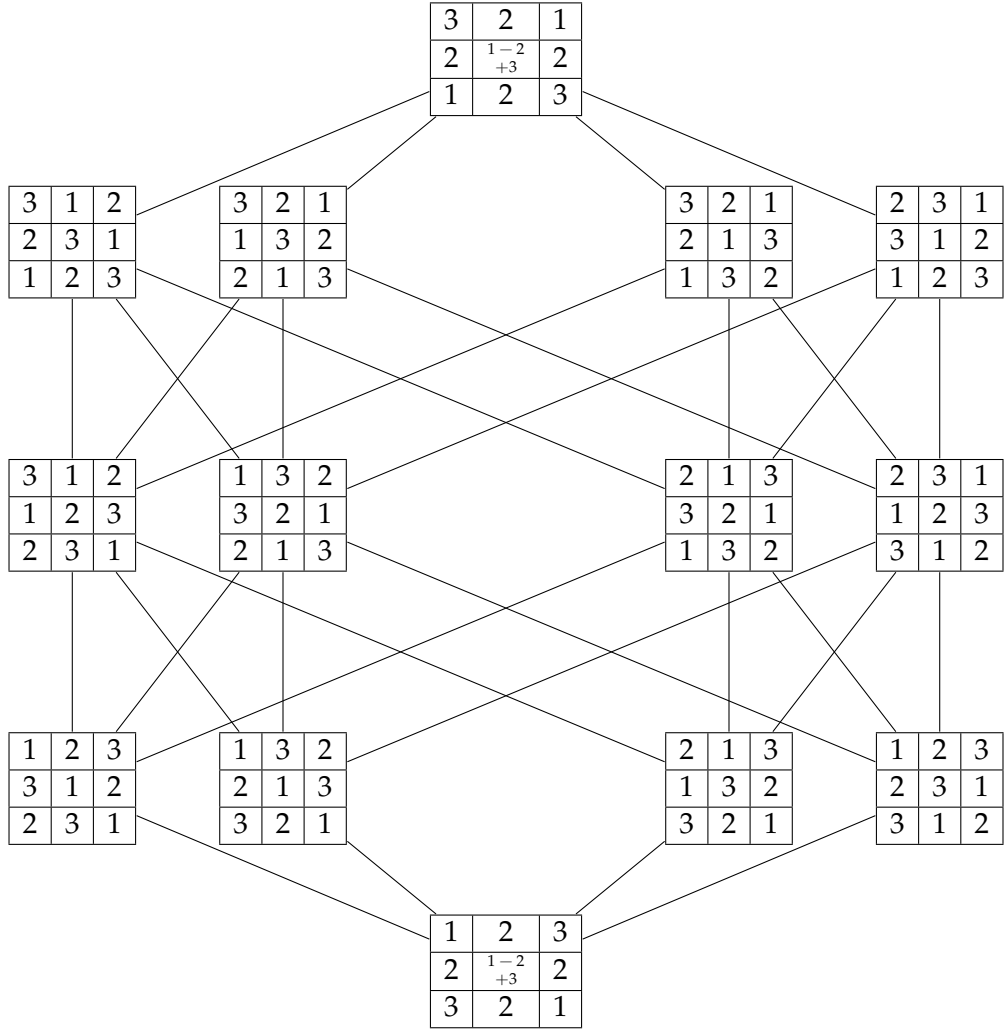

\begin{center}\quad\quad\quad \begin{NiceTabular}{ccccc}
&\phantom{$\begin{pmatrix}0&0&0&0&0\\0&0&0&0&0\\0&0&0&0&0\\0&0&0&0&0\\0&0&0&0&0\\0&0&0&0&0\\0&0&0&0&0\end{pmatrix}$}&$\begin{array}{|c|c|c|}
\hline
3	&			2	&	1\\
\hline
2	&\SmallArray{1-2\\+3}	&	2\\
\hline
1	&			2	&	3\\
\hline
\end{array}$&\phantom{$\begin{pmatrix}0&0&0&0&0\\0&0&0&0&0\\0&0&0&0&0\\0&0&0&0&0\\0&0&0&0&0\\0&0&0&0&0\\0&0&0&0&0\end{pmatrix}$}&

\\

$\begin{array}{|c|c|c|}
\hline
3	&	1	&	2\\
\hline
2	&	3	&	1\\
\hline
1	&	2	&	3\\
\hline
\end{array}$
&
$\begin{array}{|c|c|c|}
\hline
3	&	2	&	1\\
\hline
1	&	3	&	2\\
\hline
2	&	1	&	3\\
\hline
\end{array}$
&&
$\begin{array}{|c|c|c|}
\hline
3	&	2	&	1\\
\hline
2	&	1	&	3\\
\hline
1	&	3	&	2\\
\hline
\end{array}$
&
$\begin{array}{|c|c|c|}
\hline
2	&	3	&	1\\
\hline
3	&	1	&	2\\
\hline
1	&	2	&	3\\
\hline
\end{array}$

\\

&&\phantom{$\begin{pmatrix}0\\0\\0\\0\end{pmatrix}$}&&

\\

$\begin{array}{|c|c|c|}
\hline
3	&	1	&	2\\
\hline
1	&	2	&	3\\
\hline
2	&	3	&	1\\
\hline
\end{array}$
&
$\begin{array}{|c|c|c|}
\hline
1	&	3	&	2\\
\hline
3	&	2	&	1\\
\hline
2	&	1	&	3\\
\hline
\end{array}$
&\phantom{$\begin{pmatrix}0\\0\\0\\0\end{pmatrix}$}&
$\begin{array}{|c|c|c|}
\hline
2	&	1	&	3\\
\hline
3	&	2	&	1\\
\hline
1	&	3	&	2\\
\hline
\end{array}$
&
$\begin{array}{|c|c|c|}
\hline
2	&	3	&	1\\
\hline
1	&	2	&	3\\
\hline
3	&	1	&	2\\
\hline
\end{array}$

\\

&&\phantom{$\begin{pmatrix}0\\0\\0\\0\end{pmatrix}$}&&

\\

$\begin{array}{|c|c|c|}
\hline
1	&	2	&	3\\
\hline
3	&	1	&	2\\
\hline
2	&	3	&	1\\
\hline
\end{array}$
&
$\begin{array}{|c|c|c|}
\hline
1	&	3	&	2\\
\hline
2	&	1	&	3\\
\hline
3	&	2	&	1\\
\hline
\end{array}$
&&
$\begin{array}{|c|c|c|}
\hline
2	&	1	&	3\\
\hline
1	&	3	&	2\\
\hline
3	&	2	&	1\\
\hline
\end{array}$
&
$\begin{array}{|c|c|c|}
\hline
1	&	2	&	3\\
\hline
2	&	3	&	1\\
\hline
3	&	1	&	2\\
\hline
\end{array}$
\\

&\phantom{$\begin{pmatrix}0&0&0&0&0\\0&0&0&0&0\\0&0&0&0&0\\0&0&0&0&0\\0&0&0&0&0\\0&0&0&0&0\\0&0&0&0&0\end{pmatrix}$}&$\begin{array}{|c|c|c|}
\hline
1	&			2	&	3\\
\hline
2	&\SmallArray{1-2\\+3}	&	2\\
\hline
3	&			2	&	1\\
\hline
\end{array}$&\phantom{$\begin{pmatrix}0&0&0&0&0\\0&0&0&0&0\\0&0&0&0&0\\0&0&0&0&0\\0&0&0&0&0\\0&0&0&0&0\\0&0&0&0&0\end{pmatrix}$}&

\CodeAfter 
\tikz \draw  (1-3) -- (2-1); 
\tikz \draw  (1-3) -- (2-2); 
\tikz \draw  (1-3) -- (2-4); 
\tikz \draw  (1-3) -- (2-5); 
%%%
\tikz \draw  (2-1) -- (4-1); 
\tikz \draw  (2-1) -- (4-2); 
\tikz \draw  (2-2) -- (4-1); 
\tikz \draw  (2-2) -- (4-2); 
%%%
\tikz \draw  (2-1) -- (4-4); 
\tikz \draw  (2-2) -- (4-5);
\tikz \draw  (2-4) -- (4-1); 
\tikz \draw  (2-5) -- (4-2);  
%%%
\tikz \draw  (2-4) -- (4-4); 
\tikz \draw  (2-4) -- (4-5); 
\tikz \draw  (2-5) -- (4-4); 
\tikz \draw  (2-5) -- (4-5); 
%%%
\tikz \draw  (6-1) -- (4-1); 
\tikz \draw  (6-1) -- (4-2); 
\tikz \draw  (6-2) -- (4-1); 
\tikz \draw  (6-2) -- (4-2); 
%%%
\tikz \draw  (6-1) -- (4-4); 
\tikz \draw  (6-2) -- (4-5);
\tikz \draw  (6-4) -- (4-1); 
\tikz \draw  (6-5) -- (4-2);  
%%%
\tikz \draw  (6-4) -- (4-4); 
\tikz \draw  (6-4) -- (4-5); 
\tikz \draw  (6-5) -- (4-4); 
\tikz \draw  (6-5) -- (4-5); 
%%%
\tikz \draw  (7-3) -- (6-1); 
\tikz \draw  (7-3) -- (6-2); 
\tikz \draw  (7-3) -- (6-4); 
\tikz \draw  (7-3) -- (6-5); 
\end{NiceTabular}\end{center}
\caption{The Hasse diagram of the Bruhat order on ASHMs of order $3$.}\label{ASHM3}
\end{figure}

\section{Corner-sum hypermatrices}\label{sec:CSH}
In this section, in order to extend the Latin square (and ASHM) poset to a lattice, we define and study corner-sum hypermatrices.
\begin{definition}\label{corner-sum def}
A \emph{corner-sum hypermatrix} of order $n$ is an $(n+1)\times(n+1)\times(n+1)$ hypermatrix~$C$ with indices in $[0,n]$ where, for each fixed $0 \leq i, j \leq n$, we have $C_{i,j,0} = C_{i,0,j} = C_{0,i,j} = 0$,  $C_{i,j,n} = C_{i,n,j} = C_{n,i,j} = ij$, and the following for all $1 \leq k \leq n$.
\[C_{i,j,k}-C_{i,j,k-1},\:C_{i,k,j}-C_{i,k-1,j},\:C_{k,i,j}-C_{k-1,i,j} \in \big\{\max(0,i+j-n), \dots, \min(i,j)\big\}\]
\end{definition}

Note that $j=0$ implies $\max(0,i+j-n) = 0$ and $\min(i,j) = 0$, while $j=n$ implies $\max(0,i+j-n) = i$ and $\min(i,j) = i$. This coincides with the boundary values being fixed.

Given a hypermatrix $A$, we define  its associated corner-sum $\Xi(A)$ as follows, for all $0\leq i,j,k\leq n$.
\[\Xi(A)_{i,j,k} = \sum_{a = 1}^i \sum_{b = 1}^j \sum_{c = 1}^k A_{a,b,c}\]

We use $\Xi(L)$ for the corner-sum of the hypermatrix $H(L)$ corresponding to a Latin square~$L$. Similarly to the 2-dimensional case, the entries of a hypermatrix $A$ can be recovered from the corner-sum $C = \Xi(A)$ by $A_{i,j,k} = \Xi^{-1}(C)_{i,j,k}$, defined as follows for all $1\leq i,j,k\leq n$.
\[\Xi^{-1}(C)_{i,j,k} = C_{i,j,k}-C_{i-1,j,k}-C_{i,j-1,k}-C_{i,j,k-1}+C_{i-1,j-1,k}+C_{i-1,j,k-1}+C_{i,j-1,k-1}-C_{i-1,j-1,k-1}\]
\begin{example}\label{corner-sum hypermatrix example}
The following are a Latin square $L$, its corresponding hypermatrix $H(L)$, the associated matrix $\Sigma(L)$ and hypermatrix $\Xi(L)$.
\[L = \begin{array}{|c|c|c|}
    \hline
    1&2&3\\
    \hline
    2&3&1\\
    \hline
    3&1&2\\
    \hline
\end{array},\quad
H(L) = \begin{pmatrix}
      1 & 0 & 0 \\
      0 & 0 & 1 \\
      0 & 1 & 0\end{pmatrix}
\nearrow\begin{pmatrix}
      0 & 1 & 0 \\
      1 & 0 & 0 \\
      0 & 0 & 1\end{pmatrix}
\nearrow\begin{pmatrix}
      0 & 0 & 1 \\
      0 & 1 & 0 \\
      1 & 0 & 0\end{pmatrix},\quad\Sigma(L) = \begin{pmatrix}
    0&0&0&0\\
    0&1&3&6\\
    0&3&8&12\\
    0&6&12&18
\end{pmatrix},\]
\[\Xi(L) = \begin{pmatrix}
	0 & 0 & 0 & 0 \\
	0 & 0 & 0 & 0 \\
	0 & 0 & 0 & 0 \\
	0 & 0 & 0 & 0\end{pmatrix}
\nearrow
\begin{pmatrix}
	0 & 0 & 0 & 0 \\
	0 & 1 & 1 & 1 \\
	0 & 1 & 1 & 2 \\
	0 & 1 & 2 & 3\end{pmatrix}
\nearrow
\begin{pmatrix}
	0 & 0 & 0 & 0 \\
	0 & 1 & 2 & 2 \\
	0 & 2 & 3 & 4 \\
	0 & 2 & 4 & 6\end{pmatrix}
\nearrow
\begin{pmatrix}
	0 & 0 & 0 & 0 \\
	0 & 1 & 2 & 3 \\
	0 & 2 & 4 & 6 \\
	0 & 3 & 6 & 9\end{pmatrix}\]

We can recover the entries of $A=H(L)$ using the inverse map $\Xi^{-1}$. For example:
\begin{itemize}
\item $A_{2,1,2} = 2 - 0 - 1 - 1 + 0 + 0 + 1 - 0 = 1$
\item $A_{2,2,2} = 3 - 2 - 2 - 1 + 1 + 1 + 1 - 1 = 0$
\end{itemize}
\end{example}

We show in Theorem \ref{cornersum ASHMs} that if $A$ is an ASHM, then $\Xi(A)$ is a corner-sum hypermatrix. However, there exist hypermatrices $A$ which are not ASHMs and for which $\Xi(A)$ is a corner-sum hypermatrix. Note that the condition on the difference between consecutive entries in Definition \ref{corner-sum def} means that the minimum and maximum increases from one plane of a corner-sum hypermatrix of order $n$ to the next (in any direction) are respectively given by the corresponding entries of $\Sigma(J_n)$ and  $\Sigma(I_n)$. This means that the minimum and maximum possible increase between any pair of neighbouring entries in a corner-sum hypermatrix is achieved by $\Xi(L)$ for some Latin square $L$ (alternatively, by $\Xi(H(L))$ for some permutation hypermatrix $H(L)$).

\begin{lemma}\label{cornersum permutations}
Let $A$ be a hypermatrix of order $n$. Then $A$ is a permutation hypermatrix if and only if $C = \Xi(A)$ is a corner-sum hypermatrix satisfying $\Xi^{-1}(C)_{i,j,k}\in\{0,1\}$ for all $i, j, k \in [n]$.
\end{lemma}
\begin{proof}
Suppose $A$ is a permutation hypermatrix. Then $A_{i,j,k} \in \{0, 1\}$ for all $i,j,k$.  All line sums of $A$ are 1, which implies the boundary conditions of a corner-sum hypermatrix, and each plane $P$ of $A$ (in each direction) satisfies $\Sigma(J_n) \le \Sigma(P) \le \Sigma(I_n)$ entrywise. By definition, this means that $\Xi(A)$ is a corner-sum hypermatrix.

Now suppose that $C$ is a corner-sum hypermatrix with $\Xi^{-1}(C)_{i,j,k} \in \{0,1\}$ for all $i, j, k$. The boundary conditions imply all line sums of $A=\Xi^{-1}(C)$ are 1. Therefore all planes of $A$ (in each direction) are permutation matrices, and $A$ is a permutation hypermatrix.
\end{proof}

Similarly, corner-sum hypermatrices of ASHMs can be easily characterised.

\begin{theorem}\label{cornersum ASHMs}
Let $A$ be a hypermatrix of order $n$. Then $A$ is an ASHM if and only if $\Xi(A)$ is a corner-sum hypermatrix satisfying, for all $i, j, k \in [n]$, the following conditions
\begin{align}\label{ijk1}
\Xi(A)_{i,j,k} - \Xi(A)_{i-1,j,k} - \Xi(A)_{i,j-1,k} + \Xi(A)_{i-1,j-1,k}&\in \{0, 1\}\\\label{ijk2}
\Xi(A)_{i,j,k} - \Xi(A)_{i-1,j,k} - \Xi(A)_{i,j,k-1} + \Xi(A)_{i-1,j,k-1}&\in \{0, 1\}\\\label{ijk3}
\Xi(A)_{i,j,k} - \Xi(A)_{i,j-1,k} - \Xi(A)_{i,j,k-1} + \Xi(A)_{i,j-1,k-1}&\in \{0, 1\}.\end{align}
\end{theorem}
\begin{proof}
Suppose $A$ is an ASHM and consider the $k$th plane $\Xi(A)_{**k}$ of its associated corner-sum hypermatrix. The matrix $C = \Xi(A)_{**k} - \Xi(A)_{**k-1}$ is the corner-sum matrix of the ASM occupying the $k$th plane of $A$. Since $C$ is a corner-sum matrix, $C_{i,j} - C_{i,j-1} \in \{0,1\}$. Since $C_{i,j} = \Xi(A)_{i,j,k} - \Xi(A)_{i,j,k-1}$, it follows that
\[\Xi(A)_{i,j,k} - \Xi(A)_{i,j,k-1} - \Xi(A)_{i,j-1,k} + \Xi(A)_{i,j-1,k-1} \in \{0,1\}.\]
The other conditions can be proved similarly by interchanging the roles of the $i,j,k$ directions.  Since the difference between any pair of adjacent planes in any direction is a corner-sum matrix, it follows that $\Xi(A)$ is a corner-sum hypermatrix.

Conversely, suppose $\Xi(A)$ is a corner-sum hypermatrix satisfying all the conditions in Eq.~\eqref{ijk1}--\eqref{ijk3} and consider its $k$th plane $\Xi(A)_{**k}$. As before, the matrix $C = \Xi(A)_{**k} - \Xi(A)_{**k-1}$ is $\Sigma(M)$ where $M$ is the matrix occupying plane $k$ of $A$. Since $\Xi(A)_{i,j,k} - \Xi(A)_{i,j,k-1} - \Xi(A)_{i,j-1,k} + \Xi(A)_{i,j-1,k-1}$ and $\Xi(A)_{i,j,k} - \Xi(A)_{i,j,k-1} - \Xi(A)_{i-1,j,k} + \Xi(A)_{i-1,j,k-1}$  are in $\{0,1\}$, it follows that $C_{i,j} - C_{i,j-1}$ and $C_{i,j} - C_{i-1,j}$ are in $\{0,1\}$ for all $1 \leq i,j \leq n$. Therefore, $C$ is a corner-sum matrix and $M$ is an ASM. The same can be proved in any direction by interchanging the roles of $i$, $j$, and $k$. Therefore $A$ is an ASHM.
\end{proof}
The following is a 3-dimensional analogue of the criterion in Equation~\eqref{eq:cornersumASM}.
\begin{theorem}\label{cornersum_implies_bruhat_entrywise}
Let $A$ and $B$ be ASHMs of order $n$. Then $A \preceq_B B$ if and only if $\Xi(A) \geq \Xi(B)$ entrywise.
\end{theorem}
\begin{proof}
 By Corollary \ref{contiguous T corollary}, $A \preceq_B B$ if and only if there exists a sequence of positive contiguous T-blocks $T_1, T_2, \dots, T_m$ such that $A = B + T_1 + T_2 + \dots + T_m$. 
Since $\Xi$ is a linear transformation, 
\[\Xi(A) = \Xi(B + T_1 + T_2 + \dots + T_m) = \Xi(B) + \Xi(T_1) + \Xi(T_2) + \dots + \Xi(T_m).\]
Now consider $\Xi(T_i)$, and recall that $T_i$ is a contiguous positive T-block. 
Then
\[\Xi(T_i) = \Xi\left(\left(\begin{array}{rr}1&-1\\-1&1\end{array}\right)\nearrow\left(\begin{array}{rr}-1&1\\1&-1\end{array}\right)\right) = \left(\begin{array}{rr}1&0\\0&0\end{array}\right)\nearrow\left(\begin{array}{rr}0&0\\0&0\end{array}\right).\]
Therefore if $A \preceq_B B$, each T-block added to $B$ to obtain $A$ adds 1 to some entry of  the corner-sum, and hence $\Xi(A) \geq \Xi(B)$. Conversely if $\Xi(A) \geq \Xi(B)$, in each position for which there is a difference $d$ in the corner-sums, adding $d$ positive contiguous T-blocks gives the required sequence of T-blocks.
\end{proof}

Recall the definition of $\sigma[i,j]$ for $\sigma\in S_n$ described at the end of Section \ref{cornersum subsection} in order to state a criterion for the Bruhat order on permutations. We now generalise this to Latin squares, starting with an analogue of $\sigma[i,j]$ for (subarrays of) Latin squares. Given a subarray $X$ of $L\in\L_n$, define $\pos(X)$ to be its set of positions and  $\sym(X)$ its set of symbols. For $(i,j)\in \pos(X)$ and $ k\in \sym(X)$, set  $$X[i,j,k] = |\{(a,b)\in([i]\times[j])\cap \pos(X):\:L_{a,b} \le k\}|.$$

For $L_1,L_2 \in \L_n$ and $i,j,k \in [0,n]$, the hypermatrix with $L_1[i,j,k]$ as its entries is equal to its corner-sum hypermatrix $\Xi(L_1)$. Therefore Theorem~\ref{cornersum_implies_bruhat_entrywise} implies that $L_1 \preceq_B L_2$ if and only if $L_1[i,j,k] \ge L_2[i,j,k]$ for all $i,j,k \in [n]$.  Note that if $L_1,L_2 \in \L_n$ differ only by subarrays $X_1$ and $X_2$ occupying the same positions and symbols, 
it suffices to check $X_1[i,j,k] \ge X_2[i,j,k]$ for all $(i,j)\in \pos(X_1)$ and $k\in \sym(X_1)$ to confirm that $L_1[i,j,k] \ge L_2[i,j,k]$ for all $i,j,k \in [n]$. This is because any other entry of $L_1$ is equal to the corresponding entry of $L_2$, and therefore these entries add the same value to $\Xi(L_1)$ and $\Xi(L_2)$.

\begin{example}\label{subarray} Consider the following Latin squares $L_1$ and $L_2$ with subarrays $X_1$ and $X_2$ occupying positions $P = \{(1,1), (1,2), (1,3), (2,1), (2,2), (3,1), (3,3)\}$ and symbols $S = [3]$.
    \[L_1 = \begin{array}{|c|c|c|c|}
\hline
1	&	2	&	3	&	4\\
\hline
2	&	3	&	4	&	1\\
\hline
3	&	4	&	1	&	2\\
\hline
4	&	1	&	2	&	3\\
\hline
\end{array}\,, \:L_2 = \begin{array}{|c|c|c|c|}
\hline
2	&	3	&	1	&	4\\
\hline
3	&	2	&	4	&	1\\
\hline
1	&	4	&	3	&	2\\
\hline
4	&	1	&	2	&	3\\
\hline
\end{array}\,, \:X_1 = \begin{array}{|c|c|c|}
\hline
1	&	2	&	3	\\
\hline
2	&	3	&		\\
\hline
3	&		&	1	\\
\hline
\end{array}\,, \: X_2 = \begin{array}{|c|c|c|}
\hline
2	&	3	&	1	\\
\hline
3	&	2	&		\\
\hline
1	&		&	3	\\
\hline
\end{array}\]

The Latin squares $L_1$ and $L_2$ differ only by $X_1$ and $X_2$, and therefore it suffices to check $X_1[i,j,k] \ge X_2[i,j,k]$ for all $(i,j)\in P$ and $k\in S$ to confirm that $L_1[i,j,k] \ge L_2[i,j,k]$ for all $i,j,k \in [n]$.
\[X_1[i,j,1] = \begin{array}{|c|c|c|}
\hline
1	&	1	&	1	\\
\hline
1	&	1	&		\\
\hline
1	&		&	2	\\
\hline
\end{array}\,, \: X_1[i,j,2] = \begin{array}{|c|c|c|}
\hline
1	&	2	&	2	\\
\hline
2	&	3	&		\\
\hline
2	&		&	4	\\
\hline
\end{array}\,, \: X_1[i,j,3] = \begin{array}{|c|c|c|}
\hline
1	&	2	&	3	\\
\hline
2	&	4	&		\\
\hline
3	&		&	7	\\
\hline
\end{array}\]
\[X_2[i,j,1] = \begin{array}{|c|c|c|}
\hline
0	&	0	&	1	\\
\hline
0	&	0	&		\\
\hline
1	&		&	2	\\
\hline
\end{array}\,, \: X_2[i,j,2] = \begin{array}{|c|c|c|}
\hline
1	&	1	&	2	\\
\hline
1	&	2	&		\\
\hline
2	&		&	4	\\
\hline
\end{array}\,, \: X_2[i,j,3] = \begin{array}{|c|c|c|}
\hline
1	&	2	&	3	\\
\hline
2	&	4	&		\\
\hline
3	&		&	7	\\
\hline
\end{array}\]

Since $\Xi(L_1)_{i,j,k} = L_1[i,j,k]$ for all $i,j,k \in [0,n]$, and $X_1[i,j,k] \ge X_2[i,j,k]$ for all $(i,j)\in P$ and $k\in S$, we conclude that $L_1 \preceq_B L_2$. We generalise this idea in the following definition and theorem.
\end{example}

\begin{definition}
    Let $X$ be a subarray of $L\in\L_n$. Then $Y$ is a \emph{decreasing (resp.\ increasing) replacement} for $X$ if \begin{enumerate}
\item$\pos(X)=\pos(Y)$,
\item\label{DR} each row and column of $X$ contains exactly the same symbols as the corresponding row or column of $Y$, and 
\item $Y[i,j,k] \ge X[i,j,k]$ (resp.~$Y[i,j,k] \le X[i,j,k]$) for all $(i,j) \in  \pos(X)$ and $k \in \sym(X)$.\end{enumerate}
The array resulting from replacing $X$ in $L$ with $Y$ is denoted $L(X \mapsto Y)$.
\end{definition}

Note that by the second condition in the definition above, applying a decreasing (resp.~increasing) replacement to a subarray of a Latin square results in another Latin square. Moreover, if $L'$ is obtained from $L$ by a decreasing replacement, then $L'\preceq_B L$. The converse is also true, as stated in the following.

\begin{theorem}\label{decreasing theorem} Let $L_1,L_2\in\L_n$. Then $L_1\preceq_BL_2$ if and only if $L_1$ can be obtained from $L_2$ by a decreasing replacement.
\end{theorem}
\begin{proof}
    
    The if direction is clear.

    Now suppose $L_1 \preceq_B L_2$. Let $X$ be the subarray of $L_1$ and $Y$ be the subarray of $L_2$ occupying exactly the positions in which $L_1$ and $L_2$ differ. Since $L_1$ and $L_2$ are Latin squares, each row (column) of $X$ must contain the same symbols as the corresponding row (column) of $Y$. Theorem~\ref{cornersum_implies_bruhat_entrywise} implies that $L_1[i,j,k] \ge L_2[i,j,k]$ for all $i,j,k \in [n]$. Since $L_1$ and $L_2$ do not differ outside of $X$ and $Y$, it follows that $X[i,j,k] \ge Y[i,j,k]$ for all $(i,j)\in \pos(X)$ and $k\in \sym(X)$, and $X$ is a decreasing replacement for $Y$.
\end{proof}

Suppose $L \in \L_n$ has a subarray $X$ with decreasing replacement $Y$. If a subarray $X'$ of $L$ has a decreasing replacement $Y'$ such that $L(X' \mapsto Y')$ also has a decreasing replacement which results in $L(X\mapsto Y)$, we say that the decreasing replacement of $X$ with $Y$ is \emph{decomposable} via $Y'$. We say that the decreasing replacement is \emph{non-decomposable} otherwise. Decomposable and non-decomposable increasing replacements are defined similarly.

\begin{theorem}\label{thm:coverLS} Let $L_1,L_2\in\L_n$. Then $L_1\cover_BL_2$ if and only if $L_1$ can be obtained from $L_2$ by a non-decomposable decreasing replacement $Y$ of a subarray $X$ of $L_2$.
\end{theorem}
\begin{proof}
    Suppose $L_1\cover_BL_2$. Theorem \ref{decreasing theorem} implies that $L_1$ can be obtained from $L_2$ by a decreasing replacement $Y$ of some subarray $X$ of $L_2$. Since $L_1\cover_BL_2$, it is clear that the decreasing replacement of $X$ with $Y$ is not decomposable.
    
    Now suppose $L_1$ can be obtained from $L_2$ by a non-decomposable decreasing replacement $Y$ as in the statement of the theorem. By Theorem \ref{decreasing theorem}, this implies $L_1 \preceq_B L_2$. Now suppose there is some $L' \in \L_n$ which is not $L_1$ or $L_2$ and such that $L_1 \preceq_B L' \preceq_B L_2$. It follows from Theorem \ref{decreasing theorem} that $L'$ can be obtained by a decreasing replacement of some subarray of $L_2$, and $L_1$ can then be obtained by a decreasing replacement of some subarray of $L'$. This contradicts the fact that $Y$ is non-decomposable.
\end{proof}

Examples suggest that to check whether or not a decreasing replacement $X \mapsto Y$ is decomposable, it should be sufficient to consider decreasing replacements occupying indices and symbols between the minimum and maximum indices and symbols occupied by $X$ (i.e.~within the \emph{convex hull} of~$X$). Furthermore, for decreasing replacements of a subarray $X'$ containing $X$, we need only consider those for which there is no decreasing replacement on $X'\setminus X$ (i.e.~those for which the replacement does not consist of two commuting replacements, one of which is $X \mapsto Y$).

\medskip

\begin{example} Continuing from Example~\ref{subarray},  making the following increasing replacements from positions $(1,1)$ to $(3,3)$ in $L_1$ and $(1,1)$ to $(2,4)$ in $L_2$, respectively, results in Latin squares $L_1'$ and $L_2'$.

\[\begin{array}{|c|c|c|}
\hline
1	&		&	3	\\
\hline
	&	\phantom{3}	&		\\
\hline
3	&		&	1	\\
\hline
\end{array} \mapsto \begin{array}{|c|c|c|}
\hline
3	&		&	1	\\
\hline
	&	\phantom{3}	&		\\
\hline
1	&		&	3	\\
\hline
\end{array}: \:L_1' = \begin{array}{|c|c|c|c|}
\hline
3	&	2	&	1	&	4\\
\hline
2	&	3	&	4	&	1\\
\hline
1	&	4	&	3	&	2\\
\hline
4	&	1	&	2	&	3\\
\hline
\end{array},\:
\begin{array}{|c|c|c|c|}
\hline
2	&	3	&	1 & 4	\\
\hline
3	&	2	&	4 &1	\\
\hline
\end{array} \mapsto
\begin{array}{|c|c|c|c|}
\hline
3	&	2	&	4 &1	\\
\hline
2	&	3	&	1 & 4	\\
\hline
\end{array}:\: L_2' = \begin{array}{|c|c|c|c|}
\hline
3	&	2	&	4	&	1\\
\hline
2	&	3	&	1	&	4\\
\hline
1	&	4	&	3	&	2\\
\hline
4	&	1	&	2	&	3\\
\hline
\end{array}\]

The increasing replacement from $L_1$ to $L_1'$ is decomposable via $X_2$, and therefore $L_1'$ does not cover $L_1$. In fact, $L_1 \cover_B L_2 \cover_B L_1'$. Similarly, the increasing replacement from $L_2$ to $L_2'$ is decomposable via either of the  (commuting) replacements restricted to the upper-left or to the upper-right $2\times2$ subarray. $L_2'$ covers the two Latin squares resulting from each of these replacements, and these Latin squares cover $L_2$.
\end{example}
\medskip

We end this section with structural results about the Bruhat order on corner-sum hypermatrices. As customary, we write $x\vee y$ for the \emph{join} of $x$ and $y$ (i.e.~the least upper bound of $x$ and $y$, if it exists), and $x\wedge y$ for the \emph{meet} of $x$ and $y$ (i.e.~the greatest lower bound of $x$ and $y$, if it exists).

We denote the set of corner-sum hypermatrices of order $n$ by $\mathcal{C}_n$.

\begin{theorem}\label{thm:CSlattice}
The poset $(\C_n, \ge)$ is a lattice.
\end{theorem}
\begin{proof}
Let $C$ and $D$ be two corner-sum hypermatrices. We will prove that they have a unique supremum (or join) and a unique infimum (or meet), and that these are respectively $\min(C,D)$ and $\max(C,D)$ (entrywise). This implies that corner-sum hypermatrices form a lattice.

We start by showing that the hypermatrices obtained by taking entrywise minima and maxima are corner-sum hypermatrices. Clearly, the boundary conditions of a corner-sum hypermatrix are satisfied by $\min(C, D)$ and $\max(C,D)$, as these entries coincide in $C$ and $D$.

Consider the difference between any consecutive entries in $\min(C, D)$. This is bounded above and below according to the definition of corner-sum hypermatrices. If the corresponding entries of $C$ are both less than those of $D$, then the difference in $\min(C, D)$ is equal to the difference in $C$. If the entries of $D$ are both less than those of $C$, then the difference in $\min(C, D)$ is equal to the difference in $D$. If one entry of $C$ is less and the other of $D$ is less, then the difference in $\min(C, D)$ is in between the difference in $C$ and $D$. In any case, this difference will be within the bounds required for a corner-sum hypermatrix. Similarly, the difference in $\max(C,D)$ will be within the required bounds. Therefore, both matrices are corner-sum hypermatrices.

 If $X$ is a corner-sum hypermatrix with any entry less than the corresponding entry of $\max(C,D)$, then $X$ does not dominate both $C$ and $D$. Similarly, if $X$ has any entry greater than $\min(C,D)$, then $X$ is not dominated by both $C$ and $D$. Therefore $\max(C,D)$ and $\min(C,D)$ are the unique greatest lower and least upper bounds of $C$ and $D$, respectively.
\end{proof}

We henceforth define $A \preceq_B B$ for $A,B \in \Xi^{-1}(\C_n)$ if and only if $A$ can be obtained from $B$ by the addition of positive T-blocks. Recall that Theorem~\ref{cornersum_implies_bruhat_entrywise} implies that entrywise comparison of corner-sums encodes the Bruhat order for ASHMs, which trivially extends to $\Xi^{-1}(\C_n)$. For the remainder of the paper, we use $\C_n$ to denote corner-sum hypermatrices under this order.

Note that by Corollary~\ref{contiguous T corollary} and Theorem~\ref{cornersum_implies_bruhat_entrywise},  $L_1,L_2 \in \Xi^{-1}(\C_n)$ satisfy $L_1\cover_B L_2$ if and only if they differ by a single positive contiguous T-block.
 
Recall that a lattice $\mathsf{L}$ is distributive if and only if  $x\wedge(y\vee z) = (x\wedge y)\vee (x\wedge z)$ for all $x,y,z \in \mathsf{L}$.

\begin{corollary}\label{cor:FDL} The corner-sum lattice $\mathcal C_n$ is distributive.\end{corollary} 
\begin{proof} The lattice $\C_n$ is distributive if and only if $\max(x, \min(y,z)) = \min(\max(x,y), \max(x,z))$ for all $x,y,z \in \mathbb{Z}$, which can be easily verified by separately considering the cases $x \le y \le z$, $x \le z \le y$, $y \le x \le z$, $y \le z \le x$, $z \le x \le y$, and $z \le y \le x$.
\end{proof}

\section{Rank formulae for $\C_n$}\label{sec:invrank}

In this section, we study rank properties of Latin squares and corner-sum hypermatrices under the Bruhat order, in analogy with the 2-dimensional case. 

First, note that unlike its 2-dimensional analogue, the poset of Latin squares under the Bruhat order is not graded in general. For example, the following are both saturated chains in $\L_n$ between the same pair of elements.

    $$\begin{array}{|c|c|c|c|}
\hline
1	&	2	&	3	&	4\\
\hline
2	&	1	&	4  &	3\\
\hline
3	&	4	&	1	&	2\\
\hline
4	&	3	&	2	&	1\\
\hline
\end{array} \,\cover \,\begin{array}{|c|c|c|c|}
\hline
2	&	1	&	3	&	4\\
\hline
1	&	2	&	4  &	3\\
\hline
3	&	4	&	1	&	2\\
\hline
4	&	3	&	2	&	1\\
\hline
\end{array}\, \cover \, \begin{array}{|c|c|c|c|}
\hline
2	&	1	&	4	&	3\\
\hline
1	&	2	&	3  &	4\\
\hline
3	&	4	&	1	&	2\\
\hline
4	&	3	&	2	&	1\\
\hline
\end{array} \,\cover\, \begin{array}{|c|c|c|c|}
\hline
3	&	1	&	4	&	2\\
\hline
1	&	2	&	3  &	4\\
\hline
2	&	4	&	1	&	3\\
\hline
4	&	3	&	2	&	1\\
\hline
\end{array}$$ 
and
$$\begin{array}{|c|c|c|c|}
\hline
1	&	2	&	3	&	4\\
\hline
2	&	1	&	4  &	3\\
\hline
3	&	4	&	1	&	2\\
\hline
4	&	3	&	2	&	1\\
\hline
\end{array} \,\cover\, \begin{array}{|c|c|c|c|}
\hline
1	&	2	&	3	&	4\\
\hline
3	&	1	&	4  &	2\\
\hline
2	&	4	&	1	&	3\\
\hline
4	&	3	&	2	&	1\\
\hline
\end{array} \,\cover\, \begin{array}{|c|c|c|c|}
\hline
3	&	1	&	4	&	2\\
\hline
1	&	2	&	3  &	4\\
\hline
2	&	4	&	1	&	3\\
\hline
4	&	3	&	2	&	1\\
\hline
\end{array}.$$

However, Corollary~\ref{cor:FDL} implies that the corner-sum lattice $\mathcal C_n$ is  graded. This allows us to study the notions of inversions and rank in analogy with the 2-dimensional case.   
   
Recall that the rank  in $S_n$ of a permutation under the Bruhat order is equal to its number of inversions. Analogous invariants (weighted inversions \cite{weighted inversions}) provide the rank $r(A)$ of an ASM $A$ in the $n \times n$ ASM lattice, which can be inferred from the sum $\rho(A)$ of the entries in the corner-sum matrix of $A$ \cite{bruhat T}.
\[r(A) = \rho(I_n) - \rho(A) = \frac{n(n+1)(2n+1)}{6} - \rho(A)\]

Similarly, for $L \in \Xi^{-1}(\C_n)$, we define $\rho(L)$ to be the sum of the entries of the corresponding corner-sum hypermatrix $\Xi(L)$:
\[\rho(L) = \sum_{0 \leq i,j,k \leq n} \Xi(L)_{ijk}.\]
 Each cell $(i,j)$ of $L$ is a formal sum $m_1+m_2+\dots+m_l$ of elements of $\pm[n]$. Denote by $L_{ij}$ the value of the sum $m_1+m_2+\dots+m_l$. For each $k\in[l]$, there is a contribution of $+1$ if $m_k>0$ or $-1$ if $m_k<0$ to $\Xi(L)_{abc}$ for all $a \geq i, b \geq j, c \geq |m_k|$. Overall, $L_{ij}$ contributes a total of $(n-i+1)(n-j+1)(n-L_{ij}+1)$ to $\rho(L)$, which implies the following.
 \begin{equation}\label{eq:rhoL}
\rho(L) = \sum_{1 \leq i,j \leq n} (n-i+1)(n-j+1)(n-L_{ij}+1)\end{equation}
The proof of Theorem~\ref{cornersum_implies_bruhat_entrywise} implies that each T-block added to $L$ decreases $\rho(L)$ by exactly 1. Therefore the rank of a Latin square $L$ in the  corner-sum Latin lattice is $r(L) = \rho(\Xi^{-1}(M_n)) - \rho(L)$, where $M_n$ is the minimal element of the lattice $\mathcal{C}_n$, characterised in the following.

\begin{lemma}\label{minimal_element}
    Let $M_n$ be the corner-sum hypermatrix whose entries are defined by
    \begin{equation}\label{eq:Mn}
    (M_n)_{i,j,k} = \min\big(k\min(i,j),\: ij - (n-k)\max(0, i+j-n)\big).\end{equation} Then $M_n$ is the (unique) minimum element of $\mathcal{C}_n$.
\end{lemma}
\begin{proof}

    By definition, the maximum possible increase in position $(i,j)$ of $M_n$ from plane $k-1$ to plane $k$ is $\min(i,j)$. If this increase occurs between every plane from plane 0 to plane $k$, then the $(i,j,k)$-entry of $M_n$ is equal to $k\min(i,j)$. Again by definition, there is also a minimum increase from plane $k$ to $k+1$, plane $k+1$ to $k+2$, $\dots$, and plane $n-1$ to $n$. Since this minimum is $\max(0, i+j-n)$ and the $(i,j,n)$-entry is equal to $ij$, it follows that the $(i,j,k)$-entry of $M_n$ must be at least $(n-k)\max(0,i+j-n)$ less than $ij$. Therefore $M_n$ achieves the maximum possible increase from each entry to the next in the $k$-direction. So each entry of $M_n$ is the maximum possible, subject to the constraints in the $k$-direction.
    
    Since the $(j,i,k)$-entry of $M_n$ is clearly equal to the $(i,j,k)$-entry, we need only show this is equal to the $(i,k,j)$-entry in order to prove each entry of $M_n$ is the maximum possible, subject to the constraints in all directions.

    The $(i,j,k)$-entry of $M_n$ depends on the relationship between $i$ and $j$ as follows.
    \begin{equation}\label{eq:minentries}
    \begin{array}{c|c|c|}
        & i\le n-j& i>n-j \\
    \hline
     i \le j   &i\min(j,k) & \min\big(ik,n^2-ni-nj-nk+ij+ik+jk\big) \\
     \hline
     i > j   &j\min(i,k) &\min\big(jk,n^2-ni-nj-nk+ij+ik+jk\big) \\
         \hline
    \end{array}\end{equation}
    To determine the value of $\min\big(ik,n^2-ni-nj-nk+ij+ik+jk\big)$, we consider when $n^2-ni-nj-nk+ij+ik+jk \ge ik$. This implies $n^2-ni-nj-nk+ij+jk \ge 0$.

    $n^2-ni-nj-nk+ij+jk = j(i+k-n) - n(i+k-n) \ge 0$ exactly when $i+k-n \le 0$. Therefore $(M_n)_{i,j,k} = ik$ for $i \le j$, $i > n-j$, and $i \le n-k$. Similarly, $(M_n)_{i,j,k} = jk$ for $i > j$, $i > n-j$, and $j \le n-k$.

    Otherwise, if $i > n-j$, it follows that $(M_n)_{i,j,k} = n^2-ni-nj-nk+ij+ik+jk$. So this occurs for $i \le j$, $i > n-j$, and $i > n-k$, or for $i > j$, $i > n-j$, and $j > n-k$.

    From the table, we see that $(M_n)_{i,j,k} = ij$ for $i \le j$, $i \le n-j$, and $j \le k$, and also for $i > j$, $i \le n-j$, and $i \le k$. Similarly, $(M_n)_{i,j,k} = ik$ for $i \le j$, $i \le n-j$, and $j > k$, and $(M_n)_{i,j,k} = jk$ for $i > j$, $i \le n-j$, and $i > k$.

    Combining cases with the same value gives the following, which is clearly equal to $(M_n)_{i,k,j}$.
    \[ (M_n)_{i,j,k} = \begin{cases} 
      ij, & i\le\min(k,n-j)\text{ and } j \le k \\
      ik, & i\le\min(j,n-k)\text{ and } j > k \\
      jk, & i>\max(j,k)\text{ and } j+k \le n \\
      n^2-ni-nj-nk+ij+ik+jk, & i>\max(n-j,n-k)\text{ and } j+k > n \\
   \end{cases}\]
\end{proof}

 By a similar argument, it can easily be shown that the $(i,j,k)$-entry of the maximum element of $\Ccal_n$ is \begin{equation}\label{eq:max}\max\big(k\max(0,i+j-n), ij-(n-k)\min(i,j)\big).\end{equation}

\begin{lemma}\label{minimal_rank} For $n \in \mathbb{N}$, the sum of the entries in the minimal element $M_n \in \C_n$ is given by the following.
\[\rho(\Xi^{-1}(M_n)) = \frac{69n^{5} + 180n^{4} + 170n^{3} + 60n^{2} + (4^{1+(-1)^n})n}{480} \]
\end{lemma}
\begin{proof}
    Let $m(n) = \rho(\Xi^{-1}(M_n)) = \sum_{0\leq i,j,k \le n} (M_n)_{i,j,k}$. Lemma \ref{minimal_element} implies
    \[m(n) = \sum_{k=0}^n\sum_{j=0}^n\sum_{i=0}^n\min\big(k\min(i,j),\: ij - (n-k)\max(0, i+j-n)\big).\]
    Fix $k$ and let $m_k(n) = \sum_{0\leq i,j \le n} (M_n)_{i,j,k}$, meaning $m(n) = \sum_{0\le k \le n} m_k(n)$. The expression for $(M_n)_{i,j,k}$ at the end of the proof of Lemma \ref{minimal_element} gives the following expression for $m_k(n)$.
    \[\begin{aligned}m_k(n) = &\sum_{j=0}^k \sum_{i=0}^{\min(k,n-j)} ij + \sum_{j=k+1}^n \sum_{i=0}^{\min(j,n-k)} ik + \sum_{j=0}^{n-k} \sum_{i=\max(j,k)+1}^{n} jk \\ &+ \sum_{j=n-k+1}^n \sum_{i=\max(n-j,n-k)+1}^{n} n^2-ni-nj-nk+ij+ik+jk\end{aligned}\]

     By shifting the indices (replacing $j$ with $n-j$ and $k$ with $n-k$), the fourth sum becomes
    \[\sum_{j=0}^{n-k-1} \sum_{i=\max(j,k)+1}^{n} ni - ij - ik + jk = \sum_{j=0}^{n-k-1} (n-j-k)\left(\sum_{i=\max(j,k)+1}^{n} i\right) + \sum_{j=0}^{n-k-1} \sum_{i=\max(j,k)+1}^{n} jk.\]
    
    There are now two sums involving $jk$ terms. Combine them in the expression for $m_k(n)$.
    \[\begin{aligned}m_k(n) = &\sum_{j=0}^k j \left(\sum_{i=0}^{\min(k,n-j)} i\right) +k \sum_{j=k+1}^n \sum_{i=0}^{\min(j,n-k)} i + 2\sum_{j=0}^{n-k-1} \sum_{i=\max(j,k)+1}^{n} jk \\ &+ k(n-k)(n-\max(k,n-k)) + \sum_{j=0}^{n-k-1} (n-j-k)\left(\sum_{i=\max(j,k)+1}^{n} i\right)\end{aligned}\]
    
    Let $T(n)$ be the $n$th triangular number, so $T(n)  = \sum_{i=0}^n i= \frac{n(n+1)}{2}$.
    \[\begin{aligned}m_k(n) = &\sum_{j=0}^k j T(\min(k,n-j)) +k \sum_{j=k+1}^n T(\min(j,n-k)) + 2k\sum_{j=0}^{n-k-1} (n-max(j,k))j \\ &+ k(n-k)\min(k,n-k) + \sum_{j=0}^{n-k-1} (n-j-k)\left(T(n)-T(\max(j,k))\right)\end{aligned}\]

    Again, split each sum into cases to evaluate the $\min$ and $\max$ functions.
    \begin{align*}m_k(n) = & \sum_{j=0}^{\min(k,n-k)} jT(k) + \sum_{j=\min(k,n-k)+1}^k jT(n-j)\\
    &+ \sum_{j=k+1}^{\max(k,n-k)} kT(j) + \sum_{j=\max(k,n-k)+1}^n kT(n-k)\\
    &+ \sum_{j=0}^{\min(k,n-k-1)} 2k(n-k)j + \sum_{j=\min(k+1,n-k)}^{n-k-1} 2k(n-j)j \\
    &+k(n-k)\min(k,n-k) + \sum_{j=0}^{n-k-1} (n-j-k)T(n)\\
    &- \sum_{j=0}^{\min(k,n-k-1)} (n-j-k)T(k) - \sum_{j=\min(k+1,n-k)}^{n-k-1} (n-j-k)T(j)\\
    = & T(\min(k,n-k))T(k) + \sum_{j=\min(k,n-k)+1}^k \frac{1}{2}\big((n^2+n)j - (2n+1)j^2 + j^3\big)\\
    &+ \frac{k}{2}\left(\sum_{j=k+1}^{\max(k,n-k)} j^2+j\right)  + k(n-\max(k,n-k))T(n-k)\\
    &+ 2k(n-k)T(\min(k,n-k-1)) + 2k\left(\sum_{j=\min(k+1,n-k)}^{n-k-1} nj-j^2\right) \\
    &+k(n-k)\max(k,n-k) + T(n)\big((n-k)^2 - T(n-k-1)\big)\\ 
    & -(n-k)\min(k+1,n-k)T(k) + T(\min(k,n-k-1))T(k)\\
    & - \frac{1}{2}\sum_{j=\min(k+1,n-k)}^{n-k-1} (n-k)j + (n-k-1)j^2 - j^3
    % &+ (n-k)(T(n)-T(k))\left(\sum_{j=0}^{\min(k,n-k-1)} 1\right) - (T(n)-T(k))\left(\sum_{j=0}^{\min(k,n-k-1)} j\right)\\
    % &+ \sum_{j=\min(k+1,n-k)}^{n-k-1} (n-k)T(n) - \frac{2T(n)+n-k}{2}j - \frac{n-k-1}{2}j^2 + \frac{1}{2}j^3 
    \end{align*}

    Let $S(n) = \sum_{j=1}^n j^2 = \frac{n(n+1)(2n+1)}{6}$, and note that $\sum_{j=1}^n j^3 = T(n)^2$.
    \[\begin{aligned}m_k(n) = & T(\min(k,n-k))T(k) + \frac{n^2+n}{2}(T(k) - T(\min(k,n-k)))\\
    &- \frac{2n+1}{2}(S(k) - S(\min(k,n-k))) + \frac{1}{2}(T(k)^2 - T(\min(k,n-k))^2)\\
    &+ \frac{k}{2}\big(S(\max(k,n-k)) - S(k) + T(\max(k,n-k)) - T(k)\big)\\
    &+ k\min(k,n-k)T(n-k) + 2knT(n-k-1) \\
    &+ 2k(n-k)T(\min(k,n-k-1)) - 2knT(\min(k,n-k-1))\big)\\
    &- 2k\big(S(n-k-1) - S(\min(k,n-k-1))\big) +k(n-k)\min(k,n-k)\\
    & + T(n)\big((n-k)^2 - T(n-k-1)\big)  -(n-k)\min(k+1,n-k)T(k)\\
    & + T(\min(k,n-k-1))T(k)  -\frac{n-k}{2}\big(T(n-k-1) - T(\min(k,n-k-1))\big)\\
    & -\frac{n-k-1}{2}\big(S(n-k-1) - S(\min(k,n-k-1))\big)\\
    & +\frac{1}{2}\big(T(n-k-1)^2 - T(\min(k,n-k-1))^2\big) \end{aligned}\]

    We now calculate $m(n)$, splitting the sum where necessary to calculate the $\max$ and $\min$ functions, since $\min(k,n-k) = \min(k-1,n-k) = n-k$ when $k > \lfloor\frac{n}{2}\rfloor$ and $\min(k+1,n-k) = n-k$ when $k > \lfloor\frac{n-1}{2}\rfloor$. Note also that $n-k$ can be swapped for $k$ in any term in the sum (since $0 \le k \le n$).

    \begin{align*}
    m(n) &= \sum_{k=0}^{n} T(n)\big(k^2 - T(k-1)\big) + 2n(n-k)T(k-1)\\
    &+ \sum_{k=0}^{\lfloor\frac{n}{2}\rfloor} T(k)^2 + k^2T(n-k) + k^2(n-k) + \frac{k}{2}\big(S(n-k) - S(k) + T(n-k) - T(k)\big)\\
    & + 2k(n-k)T(k-1) -2n(n-k)T(k-1) + T(k-1)T(n-k)\\
    &+ \sum_{k=\lfloor\frac{n+2}{2}\rfloor}^{n} T(n-k)T(k) + k(n-k)T(n-k) + k(n-k)^2 + \frac{n^2+n}{2}\big(T(k) - T(n-k)\big)\\
    &- \frac{2n+1}{2}(S(k) - S(n-k)) + \frac{1}{2}(T(k)^2 - T(n-k)^2) + T(n-k)^2\\
    & + 2k(n-k)T(n-k) -2n(n-k)T(n-k) - \left(2(n-k)+\frac{k-1}{2}\right)\big(S(k-1) - S(n-k)\big) \\
    &-\frac{k}{2}\big(T(k-1) - T(n-k)\big) +\frac{1}{2}\big(T(k-1)^2 - T(n-k)^2\big)\\
    &+ \sum_{k=0}^{\lfloor\frac{n-1}{2}\rfloor} -(n-k)(k+1)T(k) + \sum_{k=\lfloor\frac{n+1}{2}\rfloor}^{n} -(n-k)^2T(k)\\
    \end{align*}

    Multiplying, rearranging, and factorising gives the following.
    \[\begin{aligned}
    m(n) &= \sum_{k=0}^{n} nk^3 + \left(-\frac{7}{4}n^2+\frac{5}{4}n\right)k^2 + \left(n^3 - \frac{3}{4}n^2 + \frac{1}{4}n\right)k\\
    &+ \sum_{k=0}^{\lfloor\frac{n}{2}\rfloor} \left(-\frac{1}{3}\right)k^4 + \left(n-\frac{1}{2}\right)k^3 + \left(-\frac{3}{4}n^2 - \frac{3}{4}n - \frac{1}{6}\right)k^2 + \left(\frac{1}{6}n^3+\frac{5}{4}n^2+\frac{1}{12}n\right)k\\
    %&+ \sum_{k=0}^{\lfloor\frac{n}{2}\rfloor} \left(-\frac{1}{3}\right)k^4 + \left(n{\color{red}+\frac{3}{2}}\right)k^3 + \left(-\frac{3}{4}n^2 - \frac{{\color{red}15}}{4}n - \frac{1}{6}\right)k^2 + \left(\frac{1}{6}n^3+\frac{{\color{red}9}}{4}n^2+\frac{1}{12}n\right)k\\
    &+ \sum_{k=\lfloor\frac{n+2}{2}\rfloor}^{n} \left(\frac{3}{2}n+1\right)k^3 + \left(-\frac{11}{4}n^2-\frac{13}{4}n\right)k^2 + \left(\frac{3}{2}n^3+\frac{9}{4}n^2-\frac{1}{4}n\right)k + \left(-\frac{1}{4}n^4+\frac{1}{4}n^2\right)\\
    %&+\sum_{k=\lfloor\frac{n+2}{2}\rfloor}^{n} \left(\frac{3}{2}n{\color{red}-}1\right)k^3 + \left(-\frac{11}{4}n^2-\frac{{\color{red}1}}{4}n\right)k^2 + \left(\frac{3}{2}n^3+\frac{{\color{red}5}}{4}n^2-\frac{1}{4}n\right)k + \left(-\frac{1}{4}n^4+\frac{1}{4}n^2\right)\\
    &+ \sum_{k=0}^{\lfloor\frac{n-1}{2}\rfloor} \frac{1}{2}k^4 + \left(-\frac{1}{2}n+1\right)k^3 + \left(-n+\frac{1}{2}\right)k^2 + \left(-\frac{1}{2}n\right)k\\
    &+ \sum_{k=\lfloor\frac{n+1}{2}\rfloor}^{n} -\frac{1}{2}k^4 + \left(n-\frac{1}{2}\right)k^3 + \left(-\frac{1}{2}n^2+n\right)k^2 + \left(-\frac{1}{2}n^2\right)k
    \end{aligned}\]

    Let $F(n) = \sum_{k=1}^nk^4 = \frac{n(n+1)(2n+1)(3n^2+3n-1)}{30}$.
    \[\begin{aligned}
    m(n) &= nT(n)^2 + \left(-\frac{7}{4}n^2+\frac{5}{4}n\right)S(n) + \left(n^3 - \frac{3}{4}n^2 + \frac{1}{4}n\right)T(n)\\
    & -\frac{1}{3}F\left(\left\lfloor\frac{n}{2}\right\rfloor\right) + \left(n-\frac{1}{2}\right)T\left(\left\lfloor\frac{n}{2}\right\rfloor\right)^2 + \left(-\frac{3}{4}n^2 - \frac{3}{4}n - \frac{1}{6}\right)S\left(\left\lfloor\frac{n}{2}\right\rfloor\right)\\
    &+ \left(\frac{1}{6}n^3+\frac{5}{4}n^2+\frac{1}{12}n\right)T\left(\left\lfloor\frac{n}{2}\right\rfloor\right)\\
    &+ \left(\frac{3}{2}n+1\right)\left(T(n)^2 - T\left(\left\lfloor\frac{n}{2}\right\rfloor\right)^2\right) + \left(-\frac{11}{4}n^2-\frac{13}{4}n\right)\left(S(n) - S\left(\left\lfloor\frac{n}{2}\right\rfloor\right)\right)\\
    & + \left(\frac{3}{2}n^3+\frac{9}{4}n^2-\frac{1}{4}n\right)\left(T(n) - T\left(\left\lfloor\frac{n}{2}\right\rfloor\right)\right) + \left(-\frac{1}{4}n^4+\frac{1}{4}n^2\right)\left(n - \left\lfloor\frac{n}{2}\right\rfloor\right)\\
    &+ \frac{1}{2}F\left(\left\lfloor\frac{n-1}{2}\right\rfloor\right) + \left(-\frac{1}{2}n+1\right)T\left(\left\lfloor\frac{n-1}{2}\right\rfloor\right)^2 + \left(-n+\frac{1}{2}\right)S\left(\left\lfloor\frac{n-1}{2}\right\rfloor\right)\\
    &+ \left(-\frac{1}{2}n\right)T\left(\left\lfloor\frac{n-1}{2}\right\rfloor\right)\\
    & -\frac{1}{2}\left(F(n) - F\left(\left\lfloor\frac{n-1}{2}\right\rfloor\right)\right) + \left(n-\frac{1}{2}\right)\left(T(n)^2 - T\left(\left\lfloor\frac{n-1}{2}\right\rfloor\right)^2\right) \\
    & + \left(-\frac{1}{2}n^2+n\right)\left(S(n) - S\left(\left\lfloor\frac{n-1}{2}\right\rfloor\right)\right) + \left(-\frac{1}{2}n^2\right)\left(T(n) - T\left(\left\lfloor\frac{n-1}{2}\right\rfloor\right)\right)
    \end{aligned}\]

    For odd $n$, $\left\lfloor\frac{n}{2}\right\rfloor = \left\lfloor\frac{n-1}{2}\right\rfloor = \frac{n-1}{2}$ and $\left\lfloor\frac{n+1}{2}\right\rfloor = \frac{n+1}{2}$. For even $n$, $\left\lfloor\frac{n}{2}\right\rfloor = \left\lfloor\frac{n+1}{2}\right\rfloor = \frac{n}{2}$ and $\left\lfloor\frac{n-1}{2}\right\rfloor = \frac{n-2}{2}$. Substituting these in and evaluating separately gives the following.

    \[ m(n) = \begin{cases} 
          \frac{23}{160}n^{5} + \frac{3}{8}n^{4} + \frac{17}{48}n^{3} + \frac{1}{8}n^{2} + \frac{1}{480}n, & n\text{ odd} \\
          \frac{23}{160}n^{5} + \frac{3}{8}n^{4} + \frac{17}{48}n^{3} + \frac{1}{8}n^{2} + \frac{1}{30}n, & n\text{ even}
       \end{cases}\]

    These cases can be combined into one formula as follows.
    \[m(n) =  \frac{23}{160}n^{5} + \frac{3}{8}n^{4} + \frac{17}{48}n^{3} + \frac{1}{8}n^{2} + \frac{1}{64}\left(\frac{17}{15} + (-1)^n\right)n\]
    Alternatively:
    \[m(n) = \frac{138n^{5} + 360n^{4} + 340n^{3} + 120n^{2} + \left(17 + (-1)^n15\right)n}{960} \]
    Since $17+15(-1)^n = 2^{3+2(-1)^n}$:
    \[m(n) = \frac{69n^{5} + 180n^{4} + 170n^{3} + 60n^{2} + 4^{1+(-1)^n}n}{480} \]

\end{proof}

We summarise the results of Theorem \ref{cornersum_implies_bruhat_entrywise} and Lemma \ref{minimal_rank} in the following theorem.
\begin{theorem}\label{rank_corner} Let $L$ be an $n \times n$  array such that $\Xi(L) \in \mathcal{C}_n$. Then the rank $r(L)$ of $L$  in $\mathcal{C}_n$ is
\[r(L) = \frac{69n^{5} + 180n^{4} + 170n^{3} + 60n^{2} + (4^{1+(-1)^n})n}{480}   - \sum_{1 \leq i,j \leq n} (n-i+1)(n-j+1)(n-L_{ij}+1),\]
 where $L_{ij}$ is equal to the sum of the entries in cell $(i,j)$ of $L$.
\end{theorem}

 The following alternative formula for the rank of an ASM $A$ is given in \cite[Theorem~5.1]{directed graph}.
\begin{equation}\label{eq:ASM rank}
    r(A) = \frac{1}{2}\sum_{1\le i,j \le n} (i-j)^2A_{ij}
\end{equation}
 Clearly the identity $I_n$ satisfies $r(I_n)=0$ by the above formula, as that is the unique permutation matrix with $i=j$ in all positions with non-zero entries. The further from the main diagonal that a permutation matrix $P$ has a non-zero entry, the higher that entry's contribution is to the rank, and so this formula much more explicitly counts the length of a saturated chain in the ASM lattice from $I_n$ to $P$. We now adapt this formula to Latin squares  and $\Xi^{-1}(\C_n)$ more generally. 

\begin{lemma}
Let $L$ be an $n \times n$ array such that each row and column sums to $\frac{n(n+1)}{2}$. Then
\[\frac{1}{2}\sum_{1 \leq i,j \leq n} (i-j)^2(n-L_{ij}) = \frac{n^2(n+1)(n^2+n+1)}{6} - \sum_{1 \leq i,j \leq n} (n-i+1)(n-j+1)(n-L_{ij}+1).\]
\end{lemma}
\begin{proof}
\begin{align*}
\sum_{1 \leq i,j \leq n} &(n-i+1)(n-j+1)(n-L_{ij}+1) + \frac{1}{2}\sum_{1 \leq i,j \leq n} (i-j)^2(n-L_{ij})\\
	&= \sum_{1 \leq i,j \leq n} \left(\big((n+1)-i\big)\big((n+1)-j\big)\big((n+1)-L_{ij}\big)+ \frac{1}{2}(i-j)^2(n-L_{ij})\right)\\
	&= \sum_{1 \leq i,j \leq n}\bigg( (n+1)^3 - (n+1)^2(i+j+L_{ij}) + (n+1)(ij + iL_{ij} + jL_{ij}) - ijL_{ij}\\
	&\phantom{= \sum_{1 \leq i,j \leq n}}+ n\cdot\frac{i^2+j^2}{2} - nij - \frac{1}{2} (i-j)^2L_{ij} + ijL_{ij}\bigg)\\
	&= (n+1)^3(n^2) - 3(n+1)^2\cdot n \cdot \frac{n(n+1)}{2} + 3(n+1)\left(\frac{n(n+1)}{2}\right)^2\\
	&\phantom{=}+n\cdot\frac{n(n+1)(2n+1)}{6} - n\left(\frac{n(n+1)}{2}\right)^2- \frac{n(n+1)(2n+1)}{6}\cdot\frac{n(n+1)}{2}\\
	&= \frac{n^2(n+1)(n^2+n+1)}{6}
\end{align*}
\end{proof}

Combining this result with Theorem \ref{rank_corner} gives another rank formula.
\begin{corollary}\label{rank_square} Let $L$ be an $n \times n$  array such that $\Xi(L) \in \mathcal{C}_n$. Then the rank $r(L)$ of $L$  in $\mathcal{C}_n$ is
$$r(L) = \frac{-11n^{5} + 20n^{4} + 10n^{3} - 20n^{2} + (4^{1+(-1)^n})n}{480} + \frac{1}{2}\sum_{1 \leq i, j \leq n} (i-j)^2(n-L_{ij}),$$
 where $L_{ij}$ is equal to the sum of the entries in cell $(i,j)$ of $L$.
\end{corollary}

In analogy with the 2-dimensional case, this formula again computes  much more explicitly  the length of a saturated chain between a particular $L\in\Xi^{-1}(\C_n)$ and the minimal element. The lower the index of a plane, the more a deviation from the main diagonal contributes to the rank. Informally, the rank of $L$ depends on how much the first plane deviates from $I_n$, how much the second plane deviates from elements of rank $\sim \frac{1}{n-1}n$ in the ASM lattice, etc. For fixed $n$, there is a constant added to the sum, and this reflects the fact that no element of $\Xi^{-1}(\C_n)$ has each plane equal to the minimal 2-dimensional element $I_n$. As seen in Theorem~\ref{row/col/sym swap}, the 3-dimensional Bruhat order closely relies on the structure of the 2-dimensional Bruhat order. Here we see an analogous phenomenon: the rank function on $\C_n$ is closely related to that of the 2-dimensional ASM lattice.

 Given a Latin square, consider calculating the rank of the permutation implied by a row (or column or symbol) in the ASM lattice of order $n$. The sum of these ranks across all rows (or columns or symbols) of a Latin square is constant across $\L_n$. To see this, note that each Latin square in $\L_n$ is also an element of $\Xi^{-1}(\C_n)$. Theorem \ref{thm:CSlattice} and Theorem \ref{cornersum_implies_bruhat_entrywise} imply that $L_1$ can be obtained from $L_2$ by the addition of a sequence of T-blocks, both positive and negative, via their join or meet. Adding a T-block to an array in $\Xi^{-1}(\C_n)$ increases the rank of one plane by 1 and decreases another by 1, maintaining the overall sum of the ranks. This surprising fact, and the value of the sum for fixed $n$, are summarised in the following theorem.

\begin{theorem}
    For $n \in \mathbb{N}$, let $L \in\L_n$, and $P_1, P_2, \dots, P_n$ be the permutation matrices implied by the symbols (resp. rows or columns) of $L$. Then
    \begin{equation}\label{eq:rank sums equal}\sum_{k=1}^n r(P_k) = \frac{n^2(n^2-1)}{12},\end{equation}
    where $r(P_k)$ is the rank of $P_k$ in the ASM lattice of order $n$.   
\end{theorem}
\begin{proof}
    Equation~\eqref{eq:ASM rank} implies $r(P_k) = \frac{1}{2}\sum_{1 \le i, j \le n} (i-j)^2(P_k)_{ij}$. Since $\left(\sum_{k=1}^n P_k\right)_{ij} = 1$ for all $i,j\in[n]$,
    \[\sum_{k=1}^n r(P_k) = \sum_{k=1}^n\frac{1}{2}\sum_{1 \le i, j \le n} (i-j)^2(P_k)_{ij} = \frac{1}{2}\sum_{1 \le i, j \le n} (i-j)^2 = \frac{n^2(n^2-1)}{12}.\]
\end{proof}

Comparing with \cite{Latin bruhat} again, in which the partial order on a Latin square $L$ was defined by entrywise domination of the 2-dimensional corner-sum $\Sigma(L)$, we can express $\rho(L)$ in terms of the sum of all entries of $\Sigma(L)$.
\begin{theorem}\label{thm:rankL} Let $L$ be an $n \times n$  array such that $\Xi(L) \in \mathcal{C}_n$.
    \[\rho(L) + \sum_{1 \le i, j \le n} \Sigma(L)_{i,j} = \frac{n^2(n+1)^3}{4}\]
\end{theorem}
\begin{proof}
    Since $\Sigma(L)_{i,j}$ is the sum of all entries above and to the left of position $(i,j)$, it follows that $L_{ij}$ is counted in all entries of $\Sigma(L)$ from row $i$ onwards and column $j$ onwards, where $L_{ij}$ is equal to the sum of the entries in cell $(i,j)$ of $L$. Therefore $L_{ij}$ is counted $(n-i+1)(n-j+1)$ times in $\sum_{1 \le i, j \le n} \Sigma(L)_{i,j}$, giving the following.
    \[\sum_{1 \le i, j \le n} \Sigma(L)_{i,j} = \sum_{1 \le i, j \le n} (n-i+1)(n-j+1)L_{ij}\]
    Combining this with Equation~\eqref{eq:rhoL} gives the following.
    \begin{equation*}
    \begin{aligned}
       \rho(L) + \sum_{1 \le i, j \le n} \Sigma(L)_{i,j} &= \sum_{1 \leq i,j \leq n} (n-i+1)(n-j+1)\big[(n-L_{ij}+1) + L_{ij}\big]\\
       & = (n+1)\sum_{1 \leq i,j \leq n} (n-i+1)(n-j+1) \\
       &= (n+1)\sum_{i=1}^n(n-i+1)\sum_{j=1}^n(n-j+1) = (n+1)\left(\frac{n(n+1)}{2}\right)^2
    \end{aligned}
    \end{equation*}
\end{proof}

 The \emph{rank} of a graded poset $P$ is the rank of a maximal element of $P$, or equivalently the length of any maximal chain in $P$.
\begin{theorem}\label{thm:rankCn}
    The rank of $\mathcal{C}_n$ is $$\frac{9n^5 - 10n^3 + (4^{1+(-1)^n})n}{240}.$$
\end{theorem}
\begin{proof}
    The rank of $\mathcal{C}_n$ is the sum of all entries in the hypermatrix $D_n$ equal to the difference between the minimum and maximum corner-sum hypermatrices of order $n$.  By \eqref{eq:max} and \eqref{eq:Mn}, the $(i,j,k)$-entry of $D_n$ is therefore
    \[\min\big(k\min(i,j),\: ij - (n-k)\max(0, i+j-n)\big) - \max\big(k\max(0,i+j-n), ij-(n-k)\min(i,j)\big).\]

    Recall the table for the entries of $M_n$ in the proof of Lemma \ref{minimal_element}. The following is the equivalent table for the entries of the maximum element.
    \begin{equation}\label{eq:maxentries}\begin{array}{c|c|c|}
        & i\le n-j& i>n-j \\
    \hline
     i \le j   &i\max(0,j+k-n) & \max\big(k(i+j-n),i(j+k-n)\big) \\
     \hline
     i > j   &j\max(0,i+k-n) &\max\big(k(i+j-n),j(i+k-n)\big) \\
         \hline
    \end{array}\end{equation}

    Subtracting the entries of Eq.~\eqref{eq:minentries} from those of Eq.~\eqref{eq:maxentries} gives the following table for the entries of $D_n$.
    \[\begin{array}{c|c|c|}
        & i\le n-j& i>n-j \\
    \hline
     i \le j   &
     
     \begin{array}{c|c|c|}
        & j\le n-k& j>n-k \\
    \hline
     j \le k   &ij & i(n-k) \\
     \hline
     j > k   &ik &i(n-j) \\
         \hline
    \end{array}
     
     & 
     
     \begin{array}{c|c|c|}
        & i\le n-k& i>n-k \\
    \hline
     i \le k   &k(n-j) & (n-i)(n-j) \\
     \hline
     i > k   &i(n-j) & (n-j)(n-k) \\
         \hline
    \end{array} 
     
     \\
     \hline
     i > j   &
     
    \begin{array}{c|c|c|}
        & i\le n-k& i>n-k \\
    \hline
     i \le k   &ij & j(n-k) \\
     \hline
     i > k   &jk & j(n-i) \\
         \hline
    \end{array}
     
     &
     
    \begin{array}{c|c|c|}
        & j\le n-k& j>n-k \\
    \hline
     j \le k   &k(n-i) & (n-i)(n-j) \\
     \hline
     j > k   & j(n-i) &(n-i)(n-k) \\
         \hline
    \end{array}
     
     \\
         \hline
    \end{array}\]

    Therefore the $(i,j,k)$-entry of $D_n$ can be rewritten as
    \[\min\big(\min(i,n-i)\min(j,n-j),\: \min(j,n-j)\min(k, n-k),\: \min(i,n-i)\min(k, n-k)\big).\]

    By symmetry, if $n$ is odd, the sum over all entries is equal to 
    \[8\sum_{0\le i,j,k \le \frac{n-1}{2}} \min(ij,ik,jk)=8\sum_{k=0}^{\frac{n-1}{2}} \left(\sum_{j=0}^k\sum_{i=0}^k ij + \sum_{j=k+1}^{\frac{n-1}{2}}\sum_{i=0}^j ik + \sum_{i=k+1}^{\frac{n-1}{2}}\sum_{j=0}^{i-1} jk\right).\]
    %\[=8\sum_{k=0}^{\frac{n-1}{2}} \Big(\frac{k^2(k+1)^2}{4} + k\sum_{j=k+1}^{\frac{n-1}{2}}\frac{j(j+1)}{2} + k\sum_{i=k+1}^{\frac{n-1}{2}}\frac{i(i+1)}{2}\Big)\]
    \[=8\sum_{k=0}^{\frac{n-1}{2}} \left(\frac{k^2(k+1)^2}{4} + \frac{k}{2}\sum_{j=k+1}^{\frac{n-1}{2}}(j^2+j) + \frac{k}{2}\sum_{i=k+1}^{\frac{n-1}{2}}(i^2-i)\right)=8\sum_{k=0}^{\frac{n-1}{2}} \left(\frac{k^2(k+1)^2}{4} + k\sum_{i=k+1}^{\frac{n-1}{2}}i^2\right)\]
    \[=8\sum_{k=0}^{\frac{n-1}{2}} \left(\frac{k^2(k+1)^2}{4} + k\left(\frac{\frac{n-1}{2}(\frac{n+1}{2})(n)}{6} - \frac{k(k+1)(2k+1)}{6}\right)\right)\]
    %\[=\sum_{k=0}^{\frac{n-1}{2}} \left(2k^2(k+1)^2 + \frac{k}{3}\left((n-1)n(n+1) - 4k(k+1)(2k+1)\right)\right)\]
    \[=\sum_{k=0}^{\frac{n-1}{2}} \left(2k^2(k+1)^2 -\frac{4k^2(k+1)(2k+1)}{3}
    +\frac{(n-1)n(n+1)}{3}k\right)\]
    \[=\frac{1}{3}\sum_{k=0}^{\frac{n-1}{2}} \left(-2k^4+2k^2+(n-1)n(n+1)k\right)\]
    \[=\frac{1}{3}\left(-2\frac{(\frac{n-1}{2})(\frac{n+1}{2})(n)(3(\frac{n-1}{2})^2+3(\frac{n-1}{2})-1)}{30}+2\frac{(\frac{n-1}{2})(\frac{n+1}{2})n}{6}+(n-1)n(n+1)\frac{(\frac{n-1}{2})(\frac{n+1}{2})}{2}\right)\]
    \[=\frac{9n^5-10n^3+n}{240}\]
    
    A similar calculation for $n$ even yields $\sum\limits_{0\le i,j,k \le n}(D_n)_{i,j,k} = \frac{9n^5-10n^3+16n}{240}$.
\end{proof}
It is an open question whether the previous result and the structure of the lattice $\mathcal C _n$ can be used to improve the known bounds on the number of Latin squares of given size, cf.~Section~\ref{sec:conclusions}.
%\begin{corollary}
%    \[|\L_n| \le 2^{\frac{9n^5 - 10n^3 + (4^{1+(-1)^n})n}{240}}\]
%\end{corollary}}

\section{The corner-sum Latin lattice}\label{sec:DM}
We now address the question of  whether or not corner-sum hypermatrices are the Dedekind-MacNeille completion of the Latin square poset, and we describe  properties of the set $\Xi^{-1}(\C_n)$. The lattice $\mathcal C_3$ contains 35 elements, see also Section~\ref{sec:conclusions}; its Hasse diagram is displayed in Figure~\ref{lattice3}.

 Using SageMath \cite{sage}, we have verified that the lattice $\mathcal{C}_n$ of $n \times n \times n$ corner-sum hypermatrices is the Dedekind-MacNeille completion of Latin squares of order $n$ for $n=3$. However, this is not true for $n\ge4$, as we now show.

Recall that an element in a finite lattice is \emph{join-irreducible} if and only if it covers a single element. We prove our claim by showing that $\mathcal{C}_n$ contains a join-irreducible element which is not a Latin square. Recall that $M_n$ is the minimal element in $\mathcal{C}_n$.

\begin{lemma}\label{join-irreducible}
    Let $n\ge4$, and $U_n$ be the hypermatrix which differs from $M_n$  only in the $(2,2,2)$-position by $-1$. Then $U_n$ is a corner-sum hypermatrix and $\Xi^{-1}(U_n)$ is not a Latin square.
\end{lemma}
\begin{proof}
    To prove $U_n$ is a corner-sum hypermatrix, the symmetry of $U_n$ means we only need to check $(U_n)_{2,2,2}-(U_n)_{2,2,1},(U_n)_{2,2,3}-(U_n)_{2,2,2} \in \{\max(0,2+2-n),\dots,\min(2,2)\} = \{0,1,2\}$.
    \[(M_n)_{2,2,k} = \min(k\min(2,2), 2\cdot2 - (n-k)\max(0, 2+2-n)) = \min(2k,4)\]
    Therefore $(U_n)_{2,2,1} = 2$, $(U_n)_{2,2,2} = 4-1=3$, and $(U_n)_{2,2,3} = 4$. So $(U_n)_{2,2,2}-(U_n)_{2,2,1} = (U_n)_{2,2,3}-(U_n)_{2,2,2} = 1$, and therefore $U_n$ is a corner-sum hypermatrix.

  $(U_n)_{2,2,2} = 3$, and otherwise $(U_n)_{i,j,k} = \min\big(k\min(i,j),\: ij - (n-k)\max(0, i+j-n)\big) = (M_n)_{i,j,k}$.  It is easy to see that therefore $(U_n)_{1,1,1}=1=(U_n)_{1,1,2}$ and $(U_n)_{2,2,1}=2$.

    Since $(U_n)_{2,1,1} = (U_n)_{1,2,1} = (U_n)_{1,1,2}$ and $(U_n)_{2,2,1} = (U_n)_{1,2,2} = (U_n)_{2,1,2}$, \[\Xi^{-1}(U_n)_{2,2,2} = 3 -2 -2 -2 + 1+1+1-1 = -1.\] 
    So $\Xi^{-1}(U_n)$ contains an entry not in $\{0,1\}$, and is therefore not a Latin square.
\end{proof}

\begin{theorem}
    Let $n\ge4$. Then $\mathcal{C}_n$ is not the Dedekind-MacNeille completion of $\Xi(\mathcal{L}_n)$.
\end{theorem}
\begin{proof}
    Let $n\ge4$ and let $U_n$ be the corner-sum hypermatrix defined in Lemma~\ref{join-irreducible}. Then $U_n$ is a join-irreducible element of $\mathcal{C}_n$, since it covers the minimal element, and by Lemma~\ref{join-irreducible} it is not in $\Xi(\mathcal{L}_n)$.

    This suffices to show that $\mathcal{C}_n$ is not the Dedekind-MacNeille completion of Latin squares, by the following classical argument. Recall that the image of a poset $P$ in its Dedekind-MacNeille completion is join-dense (see, e.g. \cite[Theorem 7.41]{DP}). That is, every element of said completion is the join of some subset of elements of $P$.  The element $U_n$ only covers the minimal element of $\mathcal C_n$ so it is join-irreducible but it is not the image of a Latin square, so $\Ccal_n$ is not the Dedekind-MacNeille completion of $\Xi(\mathcal{L}_n)$.\qedhere
 
\end{proof}

The lattice $\mathcal{C}_3$ contains the image under $\Xi$ of all Latin squares, ASHMs, PASHMs, and hypermatrices resulting from permuting indices in a PASHM. Additionally, $\C_3$ contains hypermatrices with entries not in $\{-1,0,1\}$, such as the following. See also Fig. \ref{lattice3}.
\[\begin{array}{|c|c|c|}
\hline
2	&	\SmallArray{1-2\\+3}	&	2\\
\hline
\SmallArray{1-2\\+3}	&	\SmallArray{-1+2+2\\+2-3}	&	\SmallArray{1-2\\+3}\\
\hline
2	&	\SmallArray{1-2\\+3}	&	2\\
\hline
\end{array} = \left(\begin{array}{rrr}0&1&0\\1&-1&1\\0&1&0\end{array}\right)
\nearrow
\left(\begin{array}{rrr}1&-1&1\\-1&3&-1\\1&-1&1\end{array}\right)
\nearrow
\left(\begin{array}{rrr}0&1&0\\1&-1&1\\0&1&0\end{array}\right)\]

\begin{lemma}\label{iff pre-image of Cn}
    Let $A$ be a hypermatrix of order $n$ and $C = \Xi(A)$. Then $C \in \mathcal{C}_n$ if and only if the corner-sum matrix of each plane $P$ of $A$ (in all three directions) satisfies $\Sigma(J_n) \le \Sigma(P) \le \Sigma(I_n)$ entrywise.
\end{lemma}
\begin{proof}
    The $(i,j)$-entry of $\Sigma(J_n)$ is $\max(0, i+j-n)$ and the $(i,j)$-entry of $\Sigma(I_n)$ is $\min(i,j)$. If $P$ is the plane with fixed $k$, then $\Sigma(P) = C_{i,j,k} - C_{i,j,k-1}$. The same is true in each direction, and the result follows trivially.
\end{proof}

While Lemma \ref{iff pre-image of Cn} provides necessary and sufficient conditions for a given hypermatrix to be in $\Xi^{-1}(\C_n)$, we now provide an alternative set of necessary, but not sufficient, conditions.

\begin{lemma}\label{line sum}
    Let $A = \Xi^{-1}(C)$, where $C \in \mathcal{C}_n$ for some $n\geq 1$. Then each line in $A$ sums to $1$.
\end{lemma}
\begin{proof}
    Fix $i,j \in [n]$. For all $k\in [n]$,
    \[C_{i,j,k} = \sum_{x,y,z = 1}^{i,j,k}A_{xyz} \implies C_{i,j,k}-C_{i-1,j,k}-C_{i,j-1,k}+C_{i-1,j-1,k} = \sum\limits_{x = 1}^{k}A_{ijx}.\]
    By definition of a corner-sum hypermatrix, we have $C_{i,j,n} = ij$. Therefore 
    \[\sum\limits_{x = 1}^{n}A_{ijx} %= C_{i,j,n}-C_{i-1,j,n}-C_{i,j-1,n}+C_{i-1,j-1,n}
    = ij - (i-1)j - i(j-1) + (i-1)(j-1) = 1.\]
    It can be similarly proven that $\sum\limits_{x = 1}^{n}A_{ixj} = \sum\limits_{x = 1}^{n}A_{xij} = 1$.
\end{proof}

\newgeometry{left=0.1cm, right = 0.1cm, top = 0.1cm, bottom = 2cm}
\begin{figure}[htb]
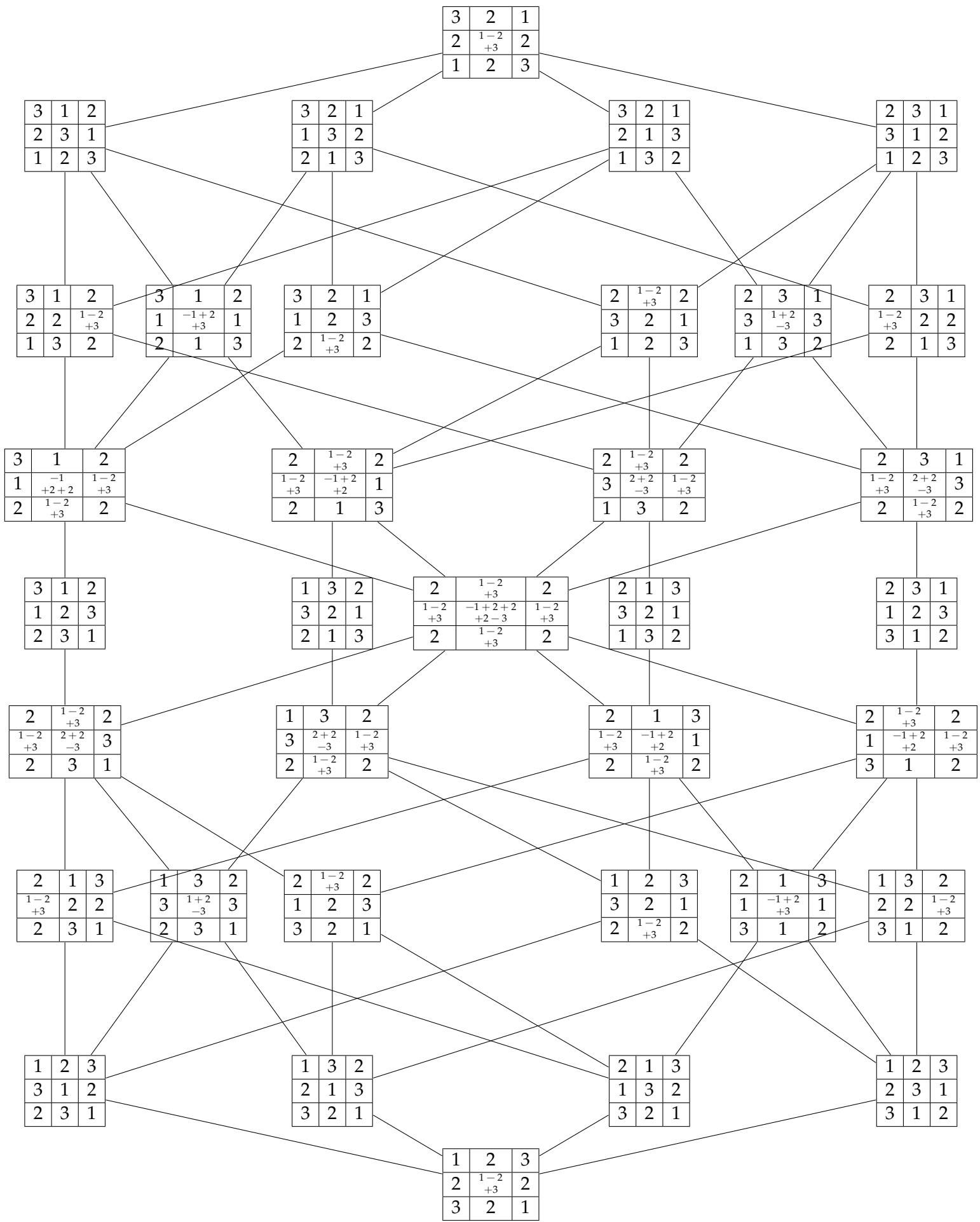
\label{lattice3}
\begin{center}\begin{NiceTabular}{ccccccc}

&&&$\begin{array}{|c|c|c|}
\hline
3	&			2	&	1\\
\hline
2	&\SmallArray{1-2\\+3}	&	2\\
\hline
1	&			2	&	3\\
\hline
\end{array}$&&&

\\

$\begin{array}{|c|c|c|}
\hline
3	&	1	&	2\\
\hline
2	&	3	&	1\\
\hline
1	&	2	&	3\\
\hline
\end{array}$
&&
$\begin{array}{|c|c|c|}
\hline
3	&	2	&	1\\
\hline
1	&	3	&	2\\
\hline
2	&	1	&	3\\
\hline
\end{array}$
&\phantom{$\begin{pmatrix}0\\0\\0\\0\\0\\\end{pmatrix}$}&
$\begin{array}{|c|c|c|}
\hline
3	&	2	&	1\\
\hline
2	&	1	&	3\\
\hline
1	&	3	&	2\\
\hline
\end{array}$
&&
$\begin{array}{|c|c|c|}
\hline
2	&	3	&	1\\
\hline
3	&	1	&	2\\
\hline
1	&	2	&	3\\
\hline
\end{array}$

\\

$\begin{array}{|c|c|c|}
\hline
	3&		1		&	2\\
\hline
2	&	2&	\SmallArray{1-2\\+3}\\
\hline
	1&			3	&2\\
\hline
\end{array}$
&$\begin{array}{|c|c|c|}
\hline
3	&			1	&	2\\
\hline
1	&\SmallArray{-1+2\\+3}	&	1\\
\hline
2	&			1	&	3\\
\hline
\end{array}$&
$\begin{array}{|c|c|c|}
\hline
	3&		2		&	1\\
\hline
	1&	2&	3\\
\hline
	2&		\SmallArray{1-2\\+3}		&	2\\
\hline
\end{array}$
&\phantom{$\begin{pmatrix}0\\0\\0\\0\\0\\0\\0\\0\\0\\0\\0\end{pmatrix}$}&
$\begin{array}{|c|c|c|}
\hline
2	&			\SmallArray{1-2\\+3}	&2	\\
\hline
	3&	2&	1\\
\hline
	1&		2		&3	\\
\hline
\end{array}$
&$\begin{array}{|c|c|c|}
\hline
2	&			3	&	1\\
\hline
3	&\SmallArray{1+2\\-3}	&	3\\
\hline
1	&			3	&	2\\
\hline
\end{array}$&
$\begin{array}{|c|c|c|}
\hline
2	&		3		&	1\\
\hline
\SmallArray{1-2\\+3}	&	2&	2\\
\hline
2	&		1		&	3\\
\hline
\end{array}$

\\

$\begin{array}{|c|c|c|}
\hline
3	&	1	&	2\\
\hline
1	&	\SmallArray{-1\\+2+2}	&	\SmallArray{1-2\\+3}\\
\hline
2	&	\SmallArray{1-2\\+3}	&	2\\
\hline
\end{array}$
&&
$\begin{array}{|c|c|c|}
\hline
2	&	\SmallArray{1-2\\+3}	&	2\\
\hline
\SmallArray{1-2\\+3}	&	\SmallArray{-1+2\\+2}	&	1\\
\hline
2	&	1	&3	\\
\hline
\end{array}$
&&
$\begin{array}{|c|c|c|}
\hline
2	&	\SmallArray{1-2\\+3}	&	2\\
\hline
3	&	\SmallArray{2+2\\-3}	&	\SmallArray{1-2\\+3}\\
\hline
1	&	3	&2	\\
\hline
\end{array}$
&&
$\begin{array}{|c|c|c|}
\hline
2	&	3	&1	\\
\hline
\SmallArray{1-2\\+3}	&	\SmallArray{2+2\\-3}	&	3\\
\hline
2	&	\SmallArray{1-2\\+3}	&2	\\
\hline
\end{array}$

\\

$\begin{array}{|c|c|c|}
\hline
3	&	1	&	2\\
\hline
1	&	2	&	3\\
\hline
2	&	3	&	1\\
\hline
\end{array}$
&\phantom{$\begin{pmatrix}0\\0\\0\\0\\0\\0\\0\\0\end{pmatrix}$}&
$\begin{array}{|c|c|c|}
\hline
1	&	3	&	2\\
\hline
3	&	2	&	1\\
\hline
2	&	1	&	3\\
\hline
\end{array}$
&$\begin{array}{|c|c|c|}
\hline
2	&	\SmallArray{1-2\\+3}	&	2\\
\hline
\SmallArray{1-2\\+3}	&	\SmallArray{-1+2+2\\+2-3}	&	\SmallArray{1-2\\+3}\\
\hline
2	&	\SmallArray{1-2\\+3}	&	2\\
\hline
\end{array}$&
$\begin{array}{|c|c|c|}
\hline
2	&	1	&	3\\
\hline
3	&	2	&	1\\
\hline
1	&	3	&	2\\
\hline
\end{array}$
&&
$\begin{array}{|c|c|c|}
\hline
2	&	3	&	1\\
\hline
1	&	2	&	3\\
\hline
3	&	1	&	2\\
\hline
\end{array}$

\\

$\begin{array}{|c|c|c|}
\hline
2	&	\SmallArray{1-2\\+3}	&2	\\
\hline
\SmallArray{1-2\\+3}	&	\SmallArray{2+2\\-3}	&3	\\
\hline
2	&	3	&	1\\
\hline
\end{array}$
&&
$\begin{array}{|c|c|c|}
\hline
	1&	3	&	2\\
\hline
3	&	\SmallArray{2+2\\-3}	&	\SmallArray{1-2\\+3}\\
\hline
2	&	\SmallArray{1-2\\+3}	&	2\\
\hline
\end{array}$
&&
$\begin{array}{|c|c|c|}
\hline
2	&	1	&	3\\
\hline
\SmallArray{1-2\\+3}	&	\SmallArray{-1+2\\+2}	&1\\
\hline
2	&	\SmallArray{1-2\\+3}	&	2\\
\hline
\end{array}$
&&
$\begin{array}{|c|c|c|}
\hline
2	&	\SmallArray{1-2\\+3}	&	2\\
\hline
1	&	\SmallArray{-1+2\\+2}	&	\SmallArray{1-2\\+3}\\
\hline
	3&		1&	2\\
\hline
\end{array}$

\\

$\begin{array}{|c|c|c|}
\hline
	2&		1		&3	\\
\hline
\SmallArray{1-2\\+3}	&2	&	2\\
\hline
	2&		3		&	1\\
\hline
\end{array}$
&$\begin{array}{|c|c|c|}
\hline
1	&			3	&	2\\
\hline
3	&\SmallArray{1+2\\-3}	&	3\\
\hline
2	&			3	&	1\\
\hline
\end{array}$&
$\begin{array}{|c|c|c|}
\hline
2	&		\SmallArray{1-2\\+3}		&	2\\
\hline
1	&	2&	3\\
\hline
3	&		2		&	1\\
\hline
\end{array}$
&\phantom{$\begin{pmatrix}0\\0\\0\\0\\0\\0\\0\\0\\0\\0\\0\end{pmatrix}$}&
$\begin{array}{|c|c|c|}
\hline
1	&		2		&	3\\
\hline
3	&	2&	1\\
\hline
	2&			\SmallArray{1-2\\+3}	&2	\\
\hline
\end{array}$
&$\begin{array}{|c|c|c|}
\hline
2	&			1	&	3\\
\hline
1	&\SmallArray{-1+2\\+3}	&	1\\
\hline
3	&			1	&	2\\
\hline
\end{array}$&
$\begin{array}{|c|c|c|}
\hline
1	&		3		&	2\\
\hline
2	&2	&	\SmallArray{1-2\\+3}\\
\hline
3	&		1		&	2\\
\hline
\end{array}$

\\

$\begin{array}{|c|c|c|}
\hline
1	&	2	&	3\\
\hline
3	&	1	&	2\\
\hline
2	&	3	&	1\\
\hline
\end{array}$
&&
$\begin{array}{|c|c|c|}
\hline
1	&	3	&	2\\
\hline
2	&	1	&	3\\
\hline
3	&	2	&	1\\
\hline
\end{array}$
&\phantom{$\begin{pmatrix}0\\0\\0\\0\\0\end{pmatrix}$}&
$\begin{array}{|c|c|c|}
\hline
2	&	1	&	3\\
\hline
1	&	3	&	2\\
\hline
3	&	2	&	1\\
\hline
\end{array}$
&&
$\begin{array}{|c|c|c|}
\hline
1	&	2	&	3\\
\hline
2	&	3	&	1\\
\hline
3	&	1	&	2\\
\hline
\end{array}$
\\

&&&$\begin{array}{|c|c|c|}
\hline
1	&			2	&	3\\
\hline
2	&\SmallArray{1-2\\+3}	&	2\\
\hline
3	&			2	&	1\\
\hline
\end{array}$&&&

\\

\CodeAfter 
\tikz \draw  (1-4) -- (2-1); 
\tikz \draw  (1-4) -- (2-3); 
\tikz \draw  (1-4) -- (2-5); 
\tikz \draw  (1-4) -- (2-7); 
%%%%%%%%%%%%%%%%%%%
\tikz \draw  (2-1) -- (3-1); 
\tikz \draw  (2-1) -- (3-2); 
\tikz \draw  (2-1) -- (3-5);
%%%
\tikz \draw  (2-3) -- (3-2); 
\tikz \draw  (2-3) -- (3-3); 
\tikz \draw  (2-3) -- (3-7); 
%%%
\tikz \draw  (2-5) -- (3-1); 
\tikz \draw  (2-5) -- (3-3); 
\tikz \draw  (2-5) -- (3-6); 
%%%
\tikz \draw  (2-7) -- (3-5); 
\tikz \draw  (2-7) -- (3-6); 
\tikz \draw  (2-7) -- (3-7); 
%%%%%%%%%%%%%%%%%%%
\tikz \draw  (3-1) -- (4-1); 
\tikz \draw  (3-1) -- (4-5); 
\tikz \draw  (3-2) -- (4-1); 
\tikz \draw  (3-2) -- (4-3); 
\tikz \draw  (3-3) -- (4-1); 
\tikz \draw  (3-3) -- (4-7); 
%%%
\tikz \draw  (3-5) -- (4-3); 
\tikz \draw  (3-5) -- (4-5); 
\tikz \draw  (3-6) -- (4-5); 
\tikz \draw  (3-6) -- (4-7); 
\tikz \draw  (3-7) -- (4-3); 
\tikz \draw  (3-7) -- (4-7); 
%%%%%%%%%%%%%%%%%%%
\tikz \draw  (4-1) -- (5-1); 
\tikz \draw  (4-3) -- (5-3); 
\tikz \draw  (4-5) -- (5-5); 
\tikz \draw  (4-7) -- (5-7); 
%%%
\tikz \draw  (4-1) -- (5-4); 
\tikz \draw  (4-3) -- (5-4); 
\tikz \draw  (4-5) -- (5-4); 
\tikz \draw  (4-7) -- (5-4); 
%%%%%%%%%%%%%%%%%%%
\tikz \draw  (5-1) -- (6-1); 
\tikz \draw  (5-3) -- (6-3); 
\tikz \draw  (5-5) -- (6-5); 
\tikz \draw  (5-7) -- (6-7); 
%%%
\tikz \draw  (5-4) -- (6-1); 
\tikz \draw  (5-4) -- (6-3); 
\tikz \draw  (5-4) -- (6-5); 
\tikz \draw  (5-4) -- (6-7); 
%%%%%%%%%%%%%%%%%%%
\tikz \draw  (6-1) -- (7-1); 
\tikz \draw  (6-1) -- (7-2); 
\tikz \draw  (6-1) -- (7-3); 
\tikz \draw  (6-3) -- (7-2); 
\tikz \draw  (6-3) -- (7-5); 
\tikz \draw  (6-3) -- (7-7); 
%%%
\tikz \draw  (6-5) -- (7-1); 
\tikz \draw  (6-5) -- (7-5); 
\tikz \draw  (6-5) -- (7-6); 
\tikz \draw  (6-7) -- (7-3); 
\tikz \draw  (6-7) -- (7-6); 
\tikz \draw  (6-7) -- (7-7); 
%%%%%%%%%%%%%%%%%%%
\tikz \draw  (7-1) -- (8-1); 
\tikz \draw  (7-1) -- (8-5); 
\tikz \draw  (7-2) -- (8-1); 
\tikz \draw  (7-2) -- (8-3); 
\tikz \draw  (7-3) -- (8-3); 
\tikz \draw  (7-3) -- (8-5); 
%%%
\tikz \draw  (7-5) -- (8-1); 
\tikz \draw  (7-5) -- (8-7); 
\tikz \draw  (7-6) -- (8-5); 
\tikz \draw  (7-6) -- (8-7); 
\tikz \draw  (7-7) -- (8-3); 
\tikz \draw  (7-7) -- (8-7); 
%%%%%%%%%%%%%%%%%%%
\tikz \draw  (9-4) -- (8-1); 
\tikz \draw  (9-4) -- (8-3); 
\tikz \draw  (9-4) -- (8-5); 
\tikz \draw  (9-4) -- (8-7); 
\end{NiceTabular}\end{center}
\caption{The Hasse diagram of  $\mathcal{C}_3$.}
\end{figure}
\restoregeometry

\begin{theorem}\label{partial sum}
   Let $A = \Xi^{-1}(C)$, where $C \in \mathcal{C}_n$ for some $n \geq 1$ and, for $i$ and $j \in [n]$, let $m_{ij} = \min(i,j,n-i+1,n-j+1)$. Then, for all $i,j,k \in [n]$,
    \[1-m_{ij} \:\:\:\:\le\:\:\:\: \sum_{x=1}^kA_{ijx},\:\:\:\:\sum_{x=1}^k A_{ixj},\:\:\:\:\sum_{x=1}^k A_{xij} \:\:\:\:\le\:\:\:\: m_{ij}.\]
\end{theorem}
\begin{proof}
    Let $\delta_{ijk} = \sum\limits_{x = 1}^{k}A_{ijx}$.
    \[C_{i,j,k} = \sum_{x,y,z = 1}^{i,j,k}A_{xyz} \implies \delta_{ijk} = C_{i,j,k}-C_{i-1,j,k}-C_{i,j-1,k}+C_{i-1,j-1,k}\]
    By definition of a corner-sum hypermatrix, $C_{i,j,k}-C_{i-1,j,k} \in \{\max(0,j+k-n),\dots,\min(j,k)\}$ and $C_{i,j-1,k}-C_{i-1,j-1,k} \in \{\max(0,j+k-n-1),\dots,\min(j-1,k)\}$. 
    
    Therefore, since $\delta_{ijk} = (C_{i,j,k}-C_{i-1,j,k})-(C_{i,j-1,k}-C_{i-1,j-1,k})$,
    \[\max(0,j+k-n)-\min(j-1,k) \le \delta_{ijk} \le \min(j,k)-\max(0,j+k-n-1).\]
    Similarly, since $\delta_{ijk} = (C_{i,j,k}-C_{i,j-1,k})-(C_{i-1,j,k}-C_{i-1,j-1,k})$,
    \[\max(0,i+k-n)-\min(i-1,k) \le \delta_{ijk} \le \min(i,k)-\max(0,i+k-n-1).\]

    Combining these gives the following inequalities.
    \[\delta_{ijk} \le \min\big(\min(j,k)-\max(0,j+k-n-1), \min(i,k)-\max(0,i+k-n-1)\big)\]
    \[\delta_{ijk} \ge \max\big(\max(0,j+k-n)-\min(j-1,k), \max(0,i+k-n)-\min(i-1,k)\big)\]

    Therefore
    \[\max_{1\le k \le n}(\delta_{ijk}) = \max_{1\le k \le n}\Big(\min\big(\min(j,k)-\max(0,j+k-n-1), \min(i,k)-\max(0,i+k-n-1)\big)\Big).\]

    Fix $i, j \le \frac{n+1}{2}$ and note $k \le n+1-\max(i,j)$ implies $j+k-n-1 \le 0$ and $i+k-n-1 \le 0$.
    
    If $1 \le k \le \min(i,j)$:
    \[\max_{1\le k \le \min(i,j)}(\delta_{ijk}) = \max_{1\le k \le \min(i,j)}\Big(\min\big(k-0, k-0\big)\Big) = \max_{1\le k \le \min(i,j)}(k) = \min(i,j)\]
    
    If $\min(i,j) \le k \le \max(i,j)$:
    \[\max_{\min(i,j)\le k \le \max(i,j)}(\delta_{ijk}) = \max_{\min(i,j)\le k \le \max(i,j)}\Big(\min\big(\min(j,k)-0, \min(i,k)-0\big)\Big)
    %\min\big(k-0, \min(i,j)-0\big)\Big) %= \max_{\min(i,j)\le k \le \max(i,j)}\big(\min(i,j)\big) 
    = \min(i,j)\]
    
    If $\max(i,j) \le k \le n+1-\max(i,j)$:
    \[\max_{\max(i,j)\le k \le n+1-\max(i,j)}(\delta_{ijk}) = \max_{\max(i,j)\le k \le n+1-\max(i,j)}\Big(\min\big(j-0, i-0\big)\Big)%= \max_{j\le k \le n+1-j}(i) 
    = \min(i,j)\]
    
   If $n+1-\max(i,j) \le k \le n+1-\min(i,j)$:
   % Your red correction above is incorrect. Whichever of $i$ or $j$ is minimal will never have the $\max(0,\dots)$ term equal anything but 0, and so you do not get cancellation of all $i$ terms \textbf{and} all $j$ terms. Hence the asymmetry in my alternative solution below. Note that after the second equal sign, rather than dealing with the $j$ situation on the left and the $i$ situation on the right, I have switched it to deal with whichever is minimal on the left and whichever is maximal on the right.
    \begin{equation*}
    \begin{aligned}
    \max_{n+1-\max(i,j)\le k \le n+1-\min(i,j)}(\delta_{ijk}) =\max_{\min(i,j)\le k \le \max(i,j)}\Big(\min\big( &\min(j, n+1-k) - \max(0, j-k),\\
    &\min(i, n+1-k) - \max(0, i-k) \big)\Big) \\
    = \max_{\min(i,j)\le k \le \max(i,j)}\Big(\min\big( \min(i,j)-0,& \max(i,j) - (\max(i,j) - k) \big)\Big) \\ = \max_{\min(i,j)\le k \le \max(i,j)}\Big(\min\big( \min(i,j), k \big)\Big)& = \min(i,j)
    \end{aligned}
    \end{equation*}
    
    If $n+1-\min(i,j) \le k \le n$:
    \[\max_{n+1-\min(i,j)\le k \le n}(\delta_{ijk}) = \max_{n+1-\min(i,j)\le k \le n}\big(n+1-k\big) = n+1-(n+1-\min(i,j)) = \min(i,j)\]

    So $\delta_{ijk} \le \min(i,j)$ for all $i, j \le \frac{n+1}{2}$, and therefore $\delta_{ijk} \le m_{ij}$ for all $i, j \le n$ by symmetry.

    Now fix $i, j \le \frac{n}{2}$, and note $k \le n-\min(i,j)$ implies $i+k-n \le 0$ and $j+k-n \le 0$. Consider
    \[\min_{1 \le k \le n}(\delta_{ijk}) = \min_{1 \le k \le n}\Big( \max\big(\max(0,j+k-n)-\min(j-1,k), \max(0,i+k-n)-\min(i-1,k)\big) \Big).\]

    If $1 \le k \le \min(i,j)-1$:
    \[\min_{1 \le k \le \min(i,j)-1}(\delta_{ijk}) = \min_{1 \le k \le \min(i,j)-1}\Big( \max\big(0-k, 0-k\big) \Big) = -\min(i,j)+1\]

    If $\min(i,j) \le k \le \max(i,j)-1$:
    \begin{equation*}
    \begin{aligned}
    \min_{\min(i,j) \le k \le \max(i,j)-1}&(\delta_{ijk}) = \min_{\min(i,j) \le k \le \max(i,j)-1}\Big( \max\big(0-\min(j-1,k), 0-\min(i-1,k)\big) \Big) \\
    & = \min_{\min(i,j) \le k \le \max(i,j)-1}\Big( \max\big(0-k, 0-(\min(i,j)-1)\big) \Big) = -\min(i,j)+1 
    \end{aligned}
    \end{equation*}

    If $\max(i,j) \le k \le n-\max(i,j)$:
    \[\min_{\max(i,j) \le k \le n-\max(i,j)}(\delta_{ijk}) = \min_{\max(i,j) \le k \le n-\max(i,j)}\Big( \max\big(0-(j-1), 0-(i-1)\big) \Big) = -\min(i,j)+1\]

 If $n-\max(i,j) \le k \le n-\min(i,j)$:
    \[\min_{n-\max(i,j)\le k \le n-\min(i,j)}(\delta_{ijk}) = \min_{n-\max(i,j) \le k \le n-\min(i,j)}\Big( \max\big(0-(j-1), 0-(i-1)\big) \Big) = -\min(i,j)+1\]

% the one immediately below resolves max(0,j+k-n) into 
 % If $n-\max(i,j) \le k \le n-\min(i,j)$:
%  \[\min_{n-\max(i,j)\le k \le n-\min(i,j)}(\delta_{ijk}) = \min_{n-\max(i,j) \le k \le n-\min(i,j)}\Big( \max\big(k-n+1, -\min(i,j)+1\big) \Big) = -\min(i,j)+1\]

    If $n-\min(i,j) \le k \le n$:
    \[\min_{n-\min(i,j) \le k \le n}(\delta_{ijk}) = \min_{n-\min(i,j) \le k \le n}\big( k-n+1\big) = -\min(i,j)+1\]

    So $\delta_{ijk} \ge -\min(i,j)+1$ for all $i,j \le \frac{n}{2}$, and therefore $\delta_{ijk} \ge -m_{ij}+1$ for all $i,j \le n$ by symmetry.
\end{proof}

\begin{corollary}\label{partial sum corollary}

    Let $A = \Xi^{-1}(C)$, where $C \in \mathcal{C}_n$ for some $n \geq 1$ and, for $i$ and $j \in [n]$, let $m_{ij} = \min(i,j,n-i+1,n-j+1)$. Then, for all $i,j,k \in [n]$,
    \[1-m_{ij} \:\:\:\:\le\:\:\:\: \sum_{x=k}^n A_{ijx},\:\:\:\:\sum_{x=k}^n A_{ixj},\:\:\:\:\sum_{x=k}^n A_{xij} \:\:\:\:\le\:\:\:\: m_{ij}.\]
\end{corollary}
\begin{proof}
    Since $1-m_{ij} \le \sum\limits_{x=1}^{k-1} A_{ijx} \le m_{ij}$ by Theorem \ref{partial sum} and $\sum\limits_{x=1}^n A_{ijx} = 1$ by Lemma \ref{line sum},
    \[1 - m_{ij} \le \sum_{x=1}^n A_{ijx} - \sum_{x=1}^{k-1} A_{ijx} \le 1 - (1-m_{ij}) = m_{ij}.\]
Each remaining inequality is proven similarly.
\end{proof}

Adding the entries within each cell of an ASHM $A\in \Xi^{-1}(\C_n)$ results in its corresponding alternating sign hypermatrix Latin-like square $L(A)$, see \cite{brualdidahl}. If $A\in\Xi^{-1}(\C_n)$ is not an ASHM, the resulting object is some other related Latin-like square.
\[L\left(\begin{array}{|c|c|c|}
\hline
1	&			2	&	3\\
\hline
2	&\SmallArray{1-2\\+3}	&	2\\
\hline
3	&			2	&	1\\
\hline
\end{array}\right) = \begin{array}{|c|c|c|}
\hline
1	&			2	&	3\\
\hline
2	&2	&	2\\
\hline
3	&			2	&	1\\
\hline
\end{array}
,\hspace{1cm}
L\left(\begin{array}{|c|c|c|}
\hline
2	&	\SmallArray{1-2\\+3}	&	2\\
\hline
\SmallArray{1-2\\+3}	&	\SmallArray{-1+2+2\\+2-3}	&	\SmallArray{1-2\\+3}\\
\hline
2	&	\SmallArray{1-2\\+3}	&	2\\
\hline
\end{array}\right) = \begin{array}{|c|c|c|}
\hline
2	&	2	&	2\\
\hline
2	&	2	&	2\\
\hline
2	&	2	&	2\\
\hline
\end{array}\]

% Made a command so we could change the symbol later if we want. Ltwo because it's the 2d Latin lattice
The lattice $\Ltwo_n$ containing these Latin-like squares is the lattice that results from applying the min/max corner-sum matrix procedure described at the end of Subsection \ref{cornersum subsection} to the poset of $n \times n$ Latin squares under the dual $\preceq'$ of the partial order described in \cite{Latin bruhat}.  Note that $\Ltwo_n = \Xi^{-1}(\C_n)$ for $n\le3$ because our definition of $\preceq_B$ coincides with $\preceq'$ for $n\le3$.

For $n\ge4$, there exist pairs of distinct ASHMs of order $n$ whose corresponding Latin-like squares coincide % $A$ and $B$ for which $L(A)=L(B)$ 
\cite{ASHM-decomp}. This means that for $n\ge4$, there are fewer elements in $\Ltwo_n$ than in $\Xi^{-1}(\C_n)$. Therefore the corner-sum Latin lattice $\Xi^{-1}(\C_n)$ is not a subset of the Latin-like square lattice $\Ltwo_n$. In addition, Example \ref{Comparing Bruhat definitions} outlines Latin squares which are comparable under $\preceq'$ and are incomparable under $\preceq_B$. Therefore the Latin-like lattice $\Ltwo_n$ is also not a subset of the corner-sum Latin lattice $\Xi^{-1}(\C_n)$. However, $\Ltwo_n$ is the image of $\C_n$ under the Latin-like square operation $L$ described above. Elements of $\C_n$ are identified and cover in $\Ltwo_n$ the union of the elements they cover in $\C_n$.

As with the lattice of corner-sum hypermatrices, $\Ltwo_n$ is not the Dedekind-MacNeille completion of $(\L_n, \preceq')$ for all $n \in \mathbb{N}$. For example, the sum of the planes of the minimal element $M_4\in\C_4$ is
\[M = \begin{pmatrix}
    4&7&9&10\\
    7&14&17&20\\
    9&17&24&30\\
    10&20&30&40
\end{pmatrix}.\]
$M$ is also the corner-sum of the minimal element of $(\L_{4}, \preceq')$. The sum of the planes of $U_4$ is 
\[X = \begin{pmatrix}
    4&7&9&10\\
    7&13&17&20\\
    9&17&24&30\\
    10&20&30&40
\end{pmatrix}.\]
This is the entrywise max of the corner-sums of the following Latin squares $A$ and $B$, and therefore is an element of $\Ltwo_n$.
\[A = \begin{array}{|c|c|c|c|}
    \hline
    4&3&2&1\\
    \hline
    3&1&4&2\\
    \hline
    2&4&1&3\\
    \hline
    1&2&3&4\\
    \hline
\end{array},\quad
B = \begin{array}{|c|c|c|c|}
    \hline
    4&2&1&3\\
    \hline
    3&4&2&1\\
    \hline
    2&1&3&4\\
    \hline
    1&3&4&2\\
    \hline
\end{array}\]
As $X$ differs from the minimal element $M$ by 1 in just one entry, it covers $M$ in $\Ltwo_4$ and is therefore join-irreducible. The pre-image of $X$ under the corner-sum operation is
\[\Sigma^{-1}(X) = \begin{array}{|c|c|c|c|}
    \hline
    4&2&1&3\\
    \hline
    3&3&2&2\\
    \hline
    2&2&3&3\\
    \hline
    1&2&3&4\\
    \hline
\end{array},\]
which is not a Latin square. Since $\Ltwo_4$ contains a join-irreducible which is not a Latin square, it is not the Dedekind-MacNeille completion of $(\L_4,\preceq')$.

\section{Monotone hypertriangles}\label{sec:MH}

 In this section, we define a 3-dimensional generalisation of \emph{monotone triangles}, prove that they are in bijection with corner-sum hypermatrices, and show that they also encode the Bruhat order by entrywise comparison.

 A monotone triangle $M$ of order $n$ is a triangular array of integers from $[n]$, arranged so that the entries of each row occur in strictly increasing order, and the entries in each north-east and south-east line are weakly increasing. We denote the $(i,j)$-entry of a monotone triangle $M$ of order $n$ with $M_{i,j}$, where we label rows 1 through $n$ from top to bottom and entries in each row from left to right. The weakly increasing condition is often reformulated as the following \emph{interlacing} condition on $M_{i,j}$, for all $i,j \in [n-1]$.
\[M_{i+1,j} \le M_{i,j} \le M_{i+1,j+1}\]

For the purpose of generalising to dimension 3, we note that a monotone triangle $M$ of order $n$ can equivalently be defined as a triangular array of integers from $[n]$ satisfying the following for all $i,j\in[n]$, where $M_{i,*}$ is row $i$ of $M$.
\begin{itemize}
    \item $M_{i,*}$ contains $i$ entries, arranged in (strictly) increasing order.
    \item Each $j\in[n]$ occurs as an entry between $0$ and $1$ times in $M_{i,*}$.
    \item From one row to the next, the increase in the number of entries $\le j$ is between $0$ and $1$.
\end{itemize}

There exists a bijection between monotone triangles and ASMs of order $n$ \cite{monotone bib}. To obtain a monotone triangle from an ASM $A$, first produce the matrix for which every entry is the sum of the entries in the corresponding column of $A$, above and including the entry itself. The positions of $+$ entries in each row of the resulting matrix indicate the entries of the corresponding row of the monotone triangle.

\[A = \begin{pmatrix}
&+&&\\
+&-&+&\\
&+&&\\
&&&+
\end{pmatrix}
\hspace{0.3cm}\rightarrow\hspace{0.3cm}
\begin{pmatrix}
&+&&\\
+&&+&\\
+&+&+&\\
+&+&+&+
\end{pmatrix}
\hspace{0.3cm}\rightarrow\hspace{0.3cm}
\begin{array}{ccccccc}
&\hspace{-0.2cm}&\hspace{-0.2cm}&\hspace{-0.2cm}2&\hspace{-0.2cm}&\hspace{-0.2cm}&\hspace{-0.2cm}\\
&\hspace{-0.2cm}&\hspace{-0.2cm}1&\hspace{-0.2cm}&\hspace{-0.2cm}3&\hspace{-0.2cm}&\hspace{-0.2cm}\\
&\hspace{-0.2cm}1&\hspace{-0.2cm}&\hspace{-0.2cm}2&\hspace{-0.2cm}&\hspace{-0.2cm}3&\hspace{-0.2cm}\\
1&\hspace{-0.2cm}&\hspace{-0.2cm}2&\hspace{-0.2cm}&\hspace{-0.2cm}3&\hspace{-0.2cm}&\hspace{-0.2cm}4
\end{array}\]

We call this intermediate matrix the \emph{partial-sum matrix} $p(A)$, and note that $p(A)$ is always a $(0,1)$-matrix for ASM $A$, since any $-1$ entry of $A$ has a $+1$ above it as the previous non-zero entry. We now define the corresponding objects for ASHMs and elements of $\Xi^{-1}(\C_n)$ more broadly.

\begin{definition}\label{monotone def}
A \emph{monotone hypertriangle} $M$ of order $n$ is a 3-dimensional array of integers from $[n]$, satisfying the following for all $i,j,k \in [n]$, where $M_{i,*,k}$ is row $i$ of plane $k$ of $M$.
\begin{enumerate}
    \item $M_{i,*,k}$ contains $ik$ entries which are arranged in (weakly) increasing order.
    \item Each $j\in[n]$ occurs as an entry between $\max(0,i+k-n)$ and $\min(i,k)$ times in $M_{i,*,k}$.
    \item From one row to the next in plane $k$ of $M$, the increase in the number of entries $\le j$ is between $\max(0,j+k-n)$ and $\min(j,k)$.
    \item From one plane to the next in row $i$ of $M$, the increase in the number of entries $\le j$ is between $\max(0,i+j-n)$ and $\min(i,j)$.
\end{enumerate}
\end{definition}

\begin{example}\label{monotone hyp}
The following is a monotone hypertriangle $M$.
\[\arraycolsep=2pt \begin{array}{ccccc}
	&	&1	&	&	\\
	&1	&	&2	&	\\
1	&	&2	&	&3	\\
\end{array}\nearrow
\begin{array}{ccccccccccc}
&&&&1&&2&&&&\\
&&1&&1&&2&&3&&\\
1&\phantom{0}&1&\phantom{0}&2&\phantom{0}&2&\phantom{0}&3&\phantom{0}&3\\
\end{array}\nearrow
\begin{array}{ccccccccccccccccc}
&&&&&&1&&2&&3&&&&&&\\
&\phantom{0}&&1&&1&&2&&2&&3&&3&&\phantom{0}&\\
1&&1&&1&&2&&2&&2&&3&&3&&3\\
\end{array}\]
In plane 2 of a monotone hypertriangle of order 3, the increase in the number of entries $\le 2$ from one row to the next is between $\max(0,2+2-3) = 1$ and $\min(2,2) = 2$. The corresponding entries of $M$ are indicated below in blue.
\[\arraycolsep=2pt \begin{array}{ccccc}
	&	&1	&	&	\\
	&1	&	&2	&	\\
1	&	&2	&	&3	\\
\end{array}\nearrow
\begin{array}{ccccccccccc}
&&&&{\color{blue}1}&&{\color{blue}2}&&&&\\
&&{\color{blue}1}&&{\color{blue}1}&&{\color{blue}2}&&3&&\\
{\color{blue}1}&\phantom{0}&{\color{blue}1}&\phantom{0}&{\color{blue}2}&\phantom{0}&{\color{blue}2}&\phantom{0}&3&\phantom{0}&3\\
\end{array}\nearrow
\begin{array}{ccccccccccccccccc}
&&&&&&1&&2&&3&&&&&&\\
&\phantom{0}&&1&&1&&2&&2&&3&&3&&\phantom{0}&\\
1&&1&&1&&2&&2&&2&&3&&3&&3\\
\end{array}\]
\end{example}

We note the following necessary (but not sufficient) interlacing conditions.

\begin{theorem}\label{interlacing}
Let $M$ be a monotone hypertriangle and $m=M_{i,j,k}$  for $i,k \in[n]$ and $j \in [ik]$. Then
\begin{itemize}
    \item $M_{i+1,j+\max(0,m+k-n),k} \le M_{i,j,k} \le M_{i+1,j+\min(m-1,k),k}$
    \item $M_{i,j+\max(0,i+m-n),k+1} \le M_{i,j,k} \le M_{i,j+\min(i,m-1),k+1}$
\end{itemize}
\end{theorem}

\begin{proof}
    We first consider the greatest $x\in\mathbb{Z}$ for which it is guaranteed that $M_{i+1,j+x,k}\le m$, since it follows that $M_{i+1,j+y,k}\le m$ for any $y\le x$. The number of entries $\le m$ in row $i+1$ is at least $\max(0, m+k-n)$ more than the number in row $i$, so any $x > \max(0, m+k-n)$ may satisfy $M_{i+1,j+x,k} > m$. Therefore $x = \max(0, m+k-n)$. %Worst case scenario is that $j$ is the greatest such that $M_{i,j,k} \le m$. 

    Similarly, we now consider the least $x\in\mathbb{Z}$ for which it is guaranteed that $m \le M_{i+1,j+x,k}$, since it follows that $m \le M_{i+1,j+y,k}$ for any $y \ge x$. The number of entries $\le m-1$ in row $i+1$ is at most $\min(m-1,k)$ more than the number in row $i$, so any $x < \min(m-1,k)$ may satisfy $m > M_{i+1,j+x,k}$. Therefore $x = \min(m-1,k)$. 
    
    The second chain of inequalities can be proved similarly. %Worst case scenario is that $j$ is the least such that $M_{i,j,k} \ge m$.
\end{proof}

For example, in plane $k=n-1$ of a monotone hypertriangle of order $n$, the first interlacing condition for $m=M_{i,j,k}=2$ from one row to the next is $M_{i+1,j+1} \le M_{i,j} \le M_{i+1,j+1}$. Equivalently, $M_{i,j} = M_{i+1,j+1}$ for all $M_{i,j}=2$ in plane $k=n-1$ of a monotone hypertriangle of order $n$. This is indicated below in blue for the monotone hypertriangle of Example \ref{monotone hyp}. For $m=3$, the condition is $M_{i+1,j+1} \le M_{i,j} \le M_{i+1,j+2}$. This is indicated below in red.
\[\begin{array}{ccccccccccc}
&&&&1&&{\color{blue}2}&&&&\\
&&1&&1&&{\color{blue}2}&&{\color{red}3}&&\\
1&\phantom{0}&1&\phantom{0}&2&\phantom{0}&{\color{blue}2}&\phantom{0}&{\color{red}3}&\phantom{0}&{\color{red}3}\\
\end{array}\]

Note that for $m>1$, the first (resp. second) interlacing condition of Theorem \ref{interlacing} for $k=1$ (resp. $i=1$) coincides with the interlacing condition of monotone triangles.

 The following are the first 4 rows and first 4 planes of a hypertriangle of order 5. If the fifth rows and fifth plane are completed to match the boundary conditions of a monotone hypertriangle, the resulting hypertriangle satisfies all interlacing conditions of Theorem \ref{interlacing}, but is not a monotone hypertriangle.  So these conditions are necessary, but not sufficient, to be a monotone hypertriangle.
\[\arraycolsep=1.4pt\def\arraystretch{1}
\begin{array}{ccccccc}
    &&&4&&&\\
    &&3&&5&&\\
    &1&&3&&5&\\
    1&&2&&4&&5
\end{array}\nearrow
\begin{array}{ccccccccccccccc}
&&&&&&3&&5&&&&&&\\
&&&&3&&3&&5&&5&&&&\\
&&1&&1&&3&&3&&5&&5&&\\
1&&1&&2&&3&&4&&4&&5&&5
\end{array}\nearrow
\begin{array}{ccccccccccccccccccccccc}
&&&&&&&&&2&&3&&5&&&&&&&&&\\
&&&&&&1&&2&&3&&3&&5&&5&&&&&&\\
&&&1&&1&&2&&3&&3&&4&&4&&5&&5&&&\\
1&\phantom{1}&1&&1&&2&&2&&3&&3&&4&&4&&4&&5&\phantom{1}&5\\
\end{array}\]\[\nearrow
\begin{array}{cccccccccccccccc}
 & & & & & &1&2&4&5& & & & & & \\
 & & & &1&1&2&3&4&4&5&5& & & & \\
 & &1&1&1&2&2&3&3&4&4&5&5&5& &\\
1&1&1&2&2&2&3&3&3&3&4&4&4&5&5&5
\end{array}
\]
Indeed, while conditions 1., 2.\ and 4.\ of Definition~\ref{monotone def} are satisfied, condition 3.\ fails in plane 2. Here, the increase in the number of entries $\le1$ from one row to the next should be between 0 and 1. However, from row 2 to row 3, there are two additional such entries.

Similarly to ASMs and monotone triangles, the set $\Xi^{-1}(\C_n)$ is in bijection with the set of monotone hypertriangles of order $n$. Before proving this, we define the partial sum hypermatrix $P(A)$ of a hypermatrix $A$ of order $n$ as follows.
\[P(A)_{i,j,k} = \sum_{a=1}^i\sum_{b=1}^k A_{a,j,b}\]
The following are a hypermatrix $A \in\Xi^{-1}(\C_n)$ and its corresponding partial-sum hypermatrix.
\begin{equation}\label{eq:hyper to triangle}
A = \left(\begin{array}{rrr}
1&0&0\\
0&1&0\\
0&0&1\\
\end{array}\right)
\nearrow
\left(\begin{array}{rrr}
0&1&0\\
1&-1&1\\
0&1&0\\
\end{array}\right)
\nearrow
\left(\begin{array}{rrr}
0&0&1\\
0&1&0\\
1&0&0\\
\end{array}\right)\end{equation}
\[P(A) = \left(\begin{array}{rrr}
1&0&0\\
1&1&0\\
1&1&1\\
\end{array}\right)
\nearrow
\left(\begin{array}{rrr}
1&1&0\\
2&1&1\\
2&2&2\\
\end{array}\right)
\nearrow
\left(\begin{array}{rrr}
1&1&1\\
2&2&2\\
3&3&3\\
\end{array}\right)\]
Let $A$ be a hypermatrix of order $n$ such that all entries of $P(A)$ are non-negative integers. We associate to $A$ a hypertriangle $\Delta(A)$, defined such that the entry $j$ of row $i$ in plane $k$ of $P(A)$ is the number of times $j$ occurs in row $i$ of plane $k$ of $\Delta(A)$. The entries in each row of $\Delta(A)$ are listed in increasing order. The hypertriangle $\Delta(A)$ of $A$ in~\eqref{eq:hyper to triangle} results in the monotone hypertriangle in Example~\ref{monotone hyp}.

\begin{lemma}
    Let $A$ be a hypermatrix of order $n$. Then $A \in \Xi^{-1}(\C_n)$ if and only if $\Delta(A)$ is a monotone hypertriangle.
\end{lemma}
\begin{proof}
    Let $C = \Xi(A)$ and $P = P(A)$. Note that
    \[C_{i,j,k} = \sum_{x=1}^j P_{i,x,k}.\]
    
    By the definition of corner-sum hypermatrices, $A \in \Xi^{-1}(\C_n)$ if and only if $P$ satisfies each of the following for all $i,j,k \in [n]$.
    \begin{enumerate}
        \item $P_{n,j,k} = k$, $P_{i,j,n} = i$, and $\sum_{x=1}^n P_{i,x,k} = ik$, by the boundary conditions.
        \item $C_{i,j,k} - C_{i,j-1,k} = \sum_{x=1}^j P_{i,x,k} - \sum_{x=1}^{j-1} P_{i,x,k} = P_{i,j,k} \in \{\max(0,i+k-n),\dots,\min(i,k)\}$.
        \item $C_{i,j,k} - C_{i-1,j,k} = \sum_{x=1}^j P_{i,x,k} - \sum_{x=1}^j P_{i-1,x,k} \in \{\max(0,j+k-n),\dots,\min(j,k)\}$.
        \item $C_{i,j,k} - C_{i,j,k-1} = \sum_{x=1}^j P_{i,x,k} - \sum_{x=1}^j P_{i,x,k-1} \in \{\max(0,i+j-n),\dots,\min(i,j)\}$.     
    \end{enumerate}
    These conditions are equivalent to the conditions of Definition \ref{monotone def}.
\end{proof}

Finally, we prove that monotone hypertriangles  also encode the Bruhat order and the lattice structure of~$\C_n$ via entrywise comparisons. Note that here the inequalities are reversed compared to Theorem~\ref{cornersum_implies_bruhat_entrywise}.

\begin{theorem}\label{partialsum_implies_bruhat_entrywise}
Let $A,B \in \Xi^{-1}(\C_n)$. Then $A \preceq_B B$ if and only if $\Delta(A)_{i,j,k} \le \Delta(B)_{i,j,k}$ entrywise.
\end{theorem}
\begin{proof}
By Theorem~\ref{cornersum_implies_bruhat_entrywise}, the following is true for all $i,j,k \in [n]$ if and only if $A \preceq_B B$. 
\[\sum_{x=1}^j P(A)_{i,x,k} \ge \sum_{x=1}^j P(B)_{i,x,k}\]
This means that every entry in row $i$ of plane $k$ of $P(A)$ occurs at an index less than or equal to that of $P(B)$. By definition, this means $\Delta(A)_{i,j,k} \le \Delta(B)_{i,j,k}$ entrywise.
\end{proof}

\section{Enumeration and open problems}\label{sec:conclusions}

We conclude our paper with a few questions and further research directions inspired by our study of the new Bruhat order on Latin squares. Further, we record some enumerative facts related to the objects introduced or discussed in this paper.

% \paragraph{Join-irreducibles in the corner-sum Latin lattice}
\medskip

Our definition of the Bruhat order on Latin squares and its extension to a lattice was inspired by the classical analogous result for permutations and ASMs. In that case, owing to the fact that the Bruhat order on permutations is \emph{dissective}, the lattice of ASMs arises as the Dedekind-MacNeille completion of permutations, and isomorphic, as a distributive lattice, to the poset of downsets of the join-irreducible elements of $(S_n,\preceq_B)$, cf.~\cite[Theorem~7]{reading} and \cite[Theorem~4.4]{lattice}. In particular, these join-irreducible elements have an explicit combinatorial characterisation: they are bigrassmannian permutations, i.e.~permutations with exactly one right and exactly one left \emph{descent}. The following two problems remain therefore open:
\begin{itemize}
    \item Characterise the join-irreducible elements of $\left(\Xi^{-1}(\C_n),\preceq_B\right)$.
    \item Find the Dedekind-MacNeille completion of our Bruhat order on Latin squares.
\end{itemize}

\medskip

 In Theorem \ref{thm:coverLS}, we characterise the covering relation in the Latin poset in terms of subarrays. Examples suggest that it should be possible to characterise this relation more efficiently, by restricting to non-commuting decreasing replacements containing the decreasing replacement $X$ we are considering, and restricting to the convex hull of $X$ otherwise. In Section \ref{sec:DM}, Lemma~\ref{line sum} and Theorem~\ref{partial sum} provide a set of necessary conditions for a given hypermatrix to be in $\Xi^{-1}(\C_n)$. Similarly, Theorem~\ref{interlacing} provides a set of necessary conditions for a given array to be a monotone hypertriangle. However, neither set of conditions is also sufficient. It would therefore be of interest to improve upon these results.

\medskip

A classical result due to Birkhoff-von Neumann states that the vertices of the polytope of doubly stochastic matrices with non-negative entries coincide with permutation matrices. The ASM polytope was introduced independently by Behrend and Knight~\cite{behrend} and by Striker~\cite{striker}. Striker proved the polytope is the convex hull of alternating sign matrices, and provided a description by inequalities.

In higher dimensions, it is known that an analogue of the Birkhoff-von Neumann theorem does not hold. Indeed, tristochastic tensors were shown  in \cite{linial-luria} to have vertices that are not Latin squares. However, in light of our definition of corner-sum hypermatrices and the study of the corner-sum Latin lattice, we consider a next natural step to study the geometry of the lattice $\C_n$.

%\begin{problem}
 
%\end{problem}

\medskip

We now turn our attention to some enumerative aspects related to the objects studied throughout this paper. The following table gives the known numbers of Latin squares, ASHMs, and PASHMs of order $n \leq 6$.

\[\begin{array}{r|cccccc|}
n						& 1	& 2	& 3	& 4		& 5 & 6\\
\hline
\text{Latin squares}			& 1	& 2	& 12	& 576		& 161,280 & 812,851,200 \\
\text{ASHMs} 				& 1	& 2	& 14	& 924		& 852,960 & 27,729,279,360 \\
\text{PASHMs}				& 1	& 2	& 18	& 2,424	& 14,366,880 &  \\
\text{Corner-sum hypermatrices}	& 1 	& 2	& 35	& 62,858	& & \\
\hline
\end{array}\]

The numbers of ASHMs for $n\le5$ and PASHMs for $n \le 4$ were calculated and verified by two separate computer searches; one program generated all sequences of length $n$ consisting of $n \times n$ ASMs and counted those which formed $n \times n \times n$ ASHMs or PASHMs, the other program counted the number of corner-sum hypermatrices of order $n$ that correspond to an ASHM or PASHM. The number of ASHMs of order $n=6$ was computed by Ludovic Schwob, and is available on the OEIS \cite{OEIS}. As mentioned in Section~\ref{sec:invrank}, it would be interesting to see if our study of the corner-sum Latin lattice could prove useful in providing bounds for the numbers of the various objects at play, particularly Latin squares themselves.

\subsection*{Acknowledgements}
AC would like to thank Tobias Rossmann for helpful conversations. This research was supported by Research Ireland through grant 22/FFP-P/11449.

\end{document}